\documentclass[11pt]{article}

\usepackage[english]{babel}
\usepackage{xcolor}
\usepackage{authblk}
\usepackage[T1]{fontenc}
\usepackage[utf8]{inputenc}
\usepackage{lmodern}
\usepackage{microtype}
\usepackage[a4paper,margin=27mm]{geometry}
\usepackage{amsmath,amssymb,amsthm,mathtools}
\usepackage{enumitem}
\usepackage{booktabs}
\usepackage{array}
\usepackage{xcolor}
\usepackage{hyperref}

\usepackage{hyphenat}
\hypersetup{pdftitle={Matrix Szeg\H{o} function and Matrix Orthogonal Polynomials for Multiple Cuts},
  pdfauthor={}
}

\setlist{itemsep=0.28em,topsep=0.45em}
\allowdisplaybreaks
\numberwithin{equation}{section}

\newtheorem{theorem}{Theorem}[section]
\newtheorem{proposition}[theorem]{Proposition}
\newtheorem{lemma}[theorem]{Lemma}
\newtheorem{corollary}[theorem]{Corollary}
\newtheorem{definition}[theorem]{Definition}

\newtheorem{remark}[theorem]{Remark}
\newtheorem{ass}[theorem]{Assumption}
\newtheorem{rhp}[theorem]{RH Problem}
\newcommand{\Ur}{\mathrm{U}(r)}
\newcommand{\C}{\mathbb C}
\newcommand{\R}{\mathbb R}
\newcommand{\D}{\mathbb D}
\newcommand{\T}{\mathbb T}
\newcommand{\Chat}{\widehat{\mathbb C}}
\newcommand{\Id}{I}
\newcommand{\cM}{\mathcal M}

\newcommand{\ord}{\operatorname{ord}}

\newcommand{\diag}{\operatorname{diag}}

\newcommand{\ev}{\operatorname{ev}}
\newcommand{\im}{\operatorname{Im}}
\newcommand{\Div}{\operatorname{div}}

\newcommand{\sgn}{\operatorname{sgn}}
\DeclareMathOperator{\degree}{deg}
\newcommand{\cc}{\mathsf{c}}
\renewcommand{\ss}{\mathsf{s}}

\newcommand{\Kx}{K_X}
\newcommand{\Ek}{E_K}
\newcommand{\Ch}{\widehat{\mathbb{C}}}

\title{
Matrix Szeg\H{o} Function and Matrix Orthogonal Polynomials for Multiple Cuts}
\author{Sampad Lahiry}
\affil{
Department of Mathematics, KU Leuven, Celestijnenlaan
200 B bus 2400, 3001 Leuven,  Belgium}
\affil{School of Mathematics and Statistics, The University of Melbourne, Victoria 3010, Australia}
\affil{ sampad.lahiry@kuleuven.be}
\date{}

\begin{document}
\maketitle
\begin{abstract}
    Following recent developments on matrix-valued orthogonal polynomials (MVOPs), we construct the matrix Szeg\H{o}  factorisation of the matrix weight function when it is supported on multiple intervals. We find that on the gaps the matrix function has unitary multiplicative jumps.

In the next part, performing the Deift-Zhou steepest descent analysis we study the associated global parametrix which comes from the Riemann-Hilbert problem of the MVOPs.  A solution is constructed visualising the rows of the parametrix as sections of vector bundles on a hyper-elliptic Riemann surface. An alternate view point to the global parametrix is also presented using a vanishing lemma. As an application we obtain strong asymptotics of MVOPs for multiple cuts which extends the works of Dea\~no, Kuijlaars, and Rom\'an.
\end{abstract}
\tableofcontents
\section{Introduction and statement of results}
\subsection{Result of Wiener-Masani and matrix Szeg\H{o} factorisations}
Let us fix $2g+2$ real numbers
\[
 a_0<b_0<a_1<b_1<\cdots<a_g<b_g,
\]

set \begin{equation}
    \mathcal{E}=\{a_0,b_0,\dots ,a_g,b_g\}
\end{equation}
and consider
\begin{equation}\label{eq:E-definition}
 E=\bigcup_{\ell=0}^{g}[a_\ell,b_\ell].
\end{equation}
The closed intervals \([a_\ell,b_\ell]\) will be called the
\emph{cuts}.  The \(g\) open intervals
\begin{equation}\label{eq:gaps-definition}
 \Sigma_j=(b_{j-1},a_j),\qquad j=1,\ldots,g,
\end{equation}
will be called the \emph{gaps}.
We further define the two open domains,
\begin{equation}
\begin{aligned}
 \Omega&=\Chat\setminus E,\\                            
 \Omega_0&=\Chat\setminus[a_0,b_g].
\end{aligned}
\end{equation}
It is worth noting and will be used throughout that $\Omega_0$ is simply-connected. Let
\[
 M:E\longrightarrow\C^{r\times r}
\]
be measurable and Hermitian positive
definite for every \(x\). 

We say $D: \Omega \longrightarrow \C^{r\times r}$ is the matrix Szeg\H{o} function if we have 

\begin{equation}\label{eq:goal-band}
 M(x)=D_\pm(x)D_\pm(x)^*\quad x\in E.
\end{equation}

Existence of such a factorisation is non-trivial. We will later see that such a $D$ can be multivalued. When $E$ is a single cut, i.e. $\Omega=\Omega_0$, then a factorisation is obtained using the Jukowski transform which conformally maps the slit complex plane to the unit disc, and relying on the factorisation done on the unit disc by Wiener and Masani \cite{WM}. We will need the result of Wiener-Masani so let us precisely state it.

\begin{theorem}[Wiener--Masani \cite{WM,EJL}]\label{thm:WM}
Let \(W:\T\to\C^{r\times r}\) be measurable and Hermitian positive
definite almost everywhere. Let $m$ denote Lebesgue measure on $\mathbb T$. Suppose
\[
 W_{jk}\in L^1(m),
 \qquad
 \log\det W\in L^1(m).
\]
Then there exists \(G_0\in H^2(\D)^{r\times r}\) such that
\begin{equation}\label{eq:WM-factor}
 W(\tau)=G_0(\tau)G_0(\tau)^*
 \quad\text{for almost every }\tau\in\T,
\end{equation}
and \(\det G_0\) is outer.  Any two such outer factors differ by
right multiplication by a constant unitary matrix.
\end{theorem}
We call a matrix $G$ outer if $\det G$ is outer. This also means $G$ is invertible. We refer to \cite{Nikolski19} for more details on outer functions, Hardy spaces and Smirnov class. 
We will require the following normalisation of the Wiener-Masani factorisation for the rest of the article.
\begin{lemma}\label{lem:canonical-normalization}
Every Wiener--Masani factor has a unique right-unitary multiple \(G_0\)
such that
\begin{equation}\label{eq:canonical-G0}
 G_0(0)>0,
\end{equation}
where \(>0\) means Hermitian positive definite.
\end{lemma}
\begin{proof}
Since \(\det G_0\) is outer, \(G_0(0)\) is invertible.  Use its
\emph{left polar decomposition}
\[
 G_0(0)=PW,
 \qquad
 P=(G_0(0)G_0(0)^*)^{1/2}>0,
 \qquad W\in U(r).
\]
Set \(G=G_0W^*\).  Then
\[
 G(0)=G_0(0)W^*=PWW^*=P>0.
\]

For uniqueness, suppose \(G_1=G_2V\), where \(V\) is unitary and both
\(G_1(0)\) and \(G_2(0)\) are positive.  The left polar decomposition
of \(G_2(0)V\) has positive factor \(G_2(0)\) and unitary factor
\(V\).  But \(G_1(0)\) is already positive, so its polar unitary is
\(\Id\).  Uniqueness of polar decomposition gives \(V=\Id\), and
therefore \(G_1=G_2\).
\end{proof}
The uniformization theorem gives a universal covering map
\begin{equation}\label{eq:universal-cover}
 p:\D\longrightarrow\Omega,
 \qquad
 p(0)=\infty.
\end{equation}
The group of deck transformations $\Gamma$ is the group of disk automorphisms
\(\gamma:\D\to\D\) satisfying
\begin{equation}\label{eq:deck}
p\circ\gamma=p.
\end{equation}
The group of deck transformation \(\Gamma\) is naturally isomorphic to
\(\pi_1(\Omega,\infty)\).

Since \(E=\partial\Omega\) has positive logarithmic capacity,
\cite[Corollary B]{FJ25} implies that \(p\) has radial limits
\[
   p^*(\tau)=\lim_{r\uparrow1}p(r\tau)\in E
\]
for almost every \(\tau\in\mathbb T\).
We can pushforward the Lebesgue measure $dm$ on $\mathbb T$ to $E$ by $p_*$, and it turns out to be the harmonic measure. \cite[(2.6)]{VY14}.
\begin{equation}\label{eq:harmonic-push}
 \int_{\T} f(p_*(\tau))\,dm(\tau)=\int_{\partial\Omega}f(\xi)\,d\omega_\infty(\xi)
\end{equation}
Thus the assumptions we need to pullback the Wiener Masani factorisation is the following.

\begin{ass}\label{ass:1}
Let $M$ be Hermitian positive definite on $E$, with
    \begin{equation}\label{eq:hyp-full}
\begin{aligned}
 &M_{jk}\in L^1(\partial\Omega,\omega_\infty)
      &&(1\leq j,k\leq r),\\
 &\log\det M\in L^1(\partial\Omega,\omega_\infty).
\end{aligned}
\end{equation}
\end{ass}

With these preliminaries we are ready to state one our main theorem.

\begin{theorem}
\label{thm:thm:WM}
    Let $M$ satisfy Assumptions \ref{ass:1}, then 

\begin{enumerate}[label=\textup{(\roman*)},leftmargin=*]

\item There exists $D:\Omega_0\longrightarrow \mathbb C^{r\times r}$ which is holomorphic and pointwise invertible on \(\Omega_0\), with
\[
 D(\infty)>0,
\]
such that
\begin{equation}\label{eq:band-main}
 M(x)=D_+(x)D_+(x)^*
     =D_-(x)D_-(x)^* 
\end{equation}
for almost every \(x \in E\).

\item For every gap $(b_{j-1},a_{j})$, there exists $U_j\in U(r)$ (unitary) such that 
\begin{equation}\label{eq:gap-main}
 D_+(x)=D_-(x)U_j, \qquad x\in (b_{j-1},a_{j}), \qquad j=1,\dots ,g.
\end{equation}
\end{enumerate}
\end{theorem}

\subsection{Application to matrix-valued orthogonal polynomials}
\label{subsec:MVOP}
 
We are interested in matrix-valued orthogonal polynomials for which the set $E$
appears either as the set of orthogonality or as the support of an equilibrium
measure.  There are two models in which this happens: in the first the weight is
fixed and the several cut structure is imposed from the outset, and in the
second the weight varies with the degree and the several cut structure is
produced by the external field.  We introduce them here, together with the
Riemann--Hilbert problems that characterize the corresponding polynomials.  In
Section \ref{subsubsec:reduction} we reduce both of them to RH Problem
\ref{def:rhp}. 
\subsubsection{Matrix-valued orthogonal polynomials}
\label{subsubsec:MVOP}
 
Let $I\subseteq\R$ be closed and of positive Lebesgue measure, and let
$W:I\to\C^{r\times r}$ be measurable.
 
\begin{ass}\label{ass:MVOP}
\begin{enumerate}
\item[(a)] $W(x)$ is Hermitian positive definite for almost every $x\in I$;
\item[(b)] $\displaystyle\int_I\|W(x)\|\,(1+|x|)^k\,dx<\infty$ for every
  $k=0,1,2,\dots$
\end{enumerate}
\end{ass}
 
Thus all moments of $W$ exist, and the moment functional is positive definite. Therefore for every $n\ge0$ a unique monic
matrix-valued orthogonal polynomial (MVOP) $P_n(x)=x^nI_r+\cdots$ of degree $n$,
characterized by
\begin{equation}\label{eq:orthogonality}
  \int_IP_n(x)\,W(x)\,P_m(x)^*\,dx=\delta_{nm}\,\mathcal H_n,
  \qquad n,m\ge0,
\end{equation}
with $\mathcal H_n$ positive definite, and these polynomials satisfy a three
term recurrence relation
\begin{equation}\label{eq:recurrence}
  xP_n(x)=P_{n+1}(x)+B_nP_n(x)+C_nP_{n-1}(x),
  \qquad P_{-1}(x)=0_r,\quad P_0(x)=I_r,
\end{equation}
with $r\times r$ matrices $B_n$ and $C_n$; see the survey \cite{DPS08}.  If
$W(x)$ is diagonal for every $x$, the MVOPs reduce to $r$ families of scalar
orthogonal polynomials.  
 
The MVOPs are characterized by a Riemann--Hilbert problem of size $2r\times2r$.
In the scalar case this is due to Fokas, Its and Kitaev \cite{FIK92}, and in the
matrix case to Cassatella-Contra and Ma\~nas \cite{CM12} and to Gr\"unbaum, de la
Iglesia and Mart\'inez-Finkelshtein \cite{GIM11}.  We state it separately for the
two models, since the jump contour and the conditions at the endpoints differ,
but in both cases the solution is
\begin{equation}\label{eq:Ysolution}
  Y(z)=
  \begin{pmatrix}
    P_n(z)&\mathcal C(P_nW)(z)\\[3pt]
    -2\pi i\,\mathcal H_{n-1}^{-1}P_{n-1}(z)&
    -2\pi i\,\mathcal H_{n-1}^{-1}\mathcal C(P_{n-1}W)(z)
  \end{pmatrix},
  \qquad
  \mathcal C(f)(z)=\frac1{2\pi i}\int_I\frac{f(s)}{s-z}\,ds,
\end{equation}
the Cauchy transform being taken entrywise.  Indeed, the jump relation follows
from $\mathcal C(f)_+=\mathcal C(f)_-+f$, and the behavior at infinity follows
from orthogonality, since the first $n$ moments of $P_nW$ vanish and the
coefficient of $z^{-n-1}$ in the expansion of $\mathcal C(P_nW)$ at infinity is
$-(2\pi i)^{-1}\mathcal H_n$.  In particular
\begin{equation}\label{eq:extractP}
  P_n(z)=\begin{pmatrix}I_r&0_r\end{pmatrix}Y(z)
         \begin{pmatrix}I_r\\0_r\end{pmatrix}
\end{equation}
is the upper left $r\times r$ block of $Y$, and $B_n$ and $C_n$ can be written
in terms of the coefficients in the expansion of $Y$ at infinity, see
\cite[\S4.1]{GIM11}.

In what follows, we use the standard notation with the Pauli matrix $\sigma_3$,
understood by blocks: for a scalar function $f(z)\ne0$ we write
\[
  f(z)^{\sigma_3}=
  \begin{pmatrix}f(z)I_r&0_r\\0_r&f(z)^{-1}I_r\end{pmatrix}.
\]
All jump contours are subsets of $\R$, oriented from left to right, and $\pm$
denotes the boundary value from $\{\pm\im z>0\}$.

\subsubsection{The first model. Jacobi type weights on \texorpdfstring{$E$}{E}}
\label{subsubsec:model1}
 
\begin{ass}\label{ass:model1}
Let $I=E$ and
\begin{equation}\label{eq:model1weight}
  W(x)=h(x)H(x),
  \qquad
  h(x)=\prod_{\ell=0}^{g}|x-a_\ell|^{\alpha_\ell}\,|x-b_\ell|^{\beta_\ell},
\end{equation}
where
\begin{enumerate}
\item[(a)] $\alpha_\ell,\beta_\ell>-1$ for $\ell=0,1,\dots,g$;
\item[(b)] $H(x)$ is an $r\times r$ complex valued matrix for $x\in E$;
\item[(c)] $H(x)$ is Hermitian positive definite for $x\in E^\circ$;
\item[(d)] $H$ is real analytic on $E$;
\item[(e)] $H(e)\neq 0$   for $e\in\mathcal E$.
\end{enumerate}
\end{ass}
 
The real analyticity in (d) means that $H$ has an analytic extension to a
neighborhood of $E$, that we also denote by $H$.  By (c) and (d) the matrices
$H(e)$ are Hermitian non-negative definite but need not be positive definite,
since eigenvalues of $H$ may vanish at the endpoints; by (e) not all of them
can.  Assumption \ref{ass:MVOP} holds, because $h>0$ on $E^\circ$ and because
$E$ is compact, $H$ is bounded and $h\in L^1(E,dx)$ by (a).  We write $\gamma_e$
for the exponent of $h$ at an endpoint,
\begin{equation}\label{eq:gamma}
  \gamma_{a_\ell}=\alpha_\ell,
  \qquad
  \gamma_{b_\ell}=\beta_\ell,
  \qquad \ell=0,1,\dots,g.
\end{equation}
 
For $g=0$ this is Assumption 1.1 of \cite{DKR23}, whose scalar antecedent is the
modified Jacobi weight of \cite{KMVV04}.  For $r=1$ and $g\ge1$ we are in the
classical situation of strong asymptotics on a system of intervals, studied by
Widom \cite{Wid69} and by Aptekarev \cite{Apt84}.
 
\begin{rhp}\label{rhp:model1}
Let $W$ be as in Assumption \ref{ass:model1} and let $n\ge1$.  We seek
$Y=Y(\cdot\,;n)$ such that
\begin{enumerate}
\item $Y:\C\setminus E\to\C^{2r\times2r}$ is analytic.
\item For $x\in E^\circ$ the matrix $Y$ admits boundary values
  $Y_\pm(x)=\lim_{\varepsilon\to0+}Y(x\pm i\varepsilon)$, related by
  \begin{equation}\label{eq:model1jump}
    Y_+(x)=Y_-(x)\begin{pmatrix}I_r&W(x)\\0_r&I_r\end{pmatrix},
    \qquad x\in(a_\ell,b_\ell),\quad\ell=0,\dots,g.
  \end{equation}
\item As $z\to\infty$,
  \begin{equation}\label{eq:model1infty}
    Y(z)=\bigl(I_{2r}+O(z^{-1})\bigr)z^{n\sigma_3}.
  \end{equation}
\item To ensure a unique solution we also specify endpoint conditions.  As
  $z\to e$ with $e\in\mathcal E$, as in \cite[\S2]{GIM11},
  \begin{equation}\label{eq:model1endpoint}
    Y(z)=\begin{pmatrix}O(1)&O(\eta_e(z))\\O(1)&O(\eta_e(z))\end{pmatrix},
    \qquad
    \eta_e(z)=
    \begin{cases}
      |z-e|^{\gamma_e},&-1<\gamma_e<0,\\
      \log|z-e|,&\gamma_e=0,\\
      1,&\gamma_e>0,
    \end{cases}
  \end{equation}
  with $\gamma_e$ as in \eqref{eq:gamma}.
\end{enumerate}
\end{rhp}
 
RH Problem \ref{rhp:model1} is solved by \eqref{eq:Ysolution} with $I=E$.  Note
that $Y$ has no jump on the gaps $\Sigma_1,\dots,\Sigma_g$, since there is no
orthogonality there.
 
\subsubsection{The second model. Varying exponential weights}
\label{subsubsec:model2}
 
\begin{ass}\label{ass:model2}
Let $I=\R$ and
\begin{equation}\label{eq:model2weight}
  W_n(x)=e^{-nV(x)}M(x),
\end{equation}
where
\begin{enumerate}
\item[(a)] $V$ is real valued and real analytic on $\R$, with
  $V(x)/\log(1+|x|^2)\to+\infty$ as $|x|\to\infty$;
\item[(b)] the equilibrium measure $\mu_V$ in the external field $V$ is
  supported on $E$ and is regular with density $\psi$;
\item[(c)] $M(x)=Q(x)Q(x)^{*}$ is an $r\times r$  matrix, Hermitian positive
  definite for $x\in\R$, and independent of $n$;
\item[(d)] $Q$ is real analytic on $E$, and $\|Q(x)\|$ grows at most
  polynomially as $|x|\to\infty$.
\end{enumerate}
\end{ass}
 
Here $\mu_V$ is the probability measure on $\R$ minimizing
$\iint\log|x-y|^{-1}d\mu(x)d\mu(y)+\int V\,d\mu$, which exists and is unique by
(a), see \cite{ST97}.  Regular means that its density is positive on $E^\circ$,
vanishes like a square root at each point of $\mathcal E$, and that the
Euler--Lagrange inequality is strict on $\R\setminus E$.  These are the standard
hypotheses of the scalar steepest descent analysis, the several cut case being
treated in \cite{DKMVZ99b}; that (b) holds generically is the content of
\cite{KM00}.  Assumption \ref{ass:MVOP} holds for every $n\ge1$, since by (a)
the factor $e^{-nV}$ decays faster than any power of $|x|^{-1}$.  Conditions (c)
and (d) say that the matrix part of the weight does not depend on the large
parameter, as in \cite[Assumption 1]{DR25}; for $g=0$ Assumption
\ref{ass:model2} is precisely the setting of \cite{DR25}, where $\mu_V$ is
assumed to be supported on a single interval.
 
\begin{rhp}\label{rhp:model2}
Let $W_n$ be as in Assumption \ref{ass:model2} and let $n\ge1$.  We seek
$Y=Y(\cdot\,;n)$ such that
\begin{enumerate}
\item $Y:\C\setminus\R\to\C^{2r\times2r}$ is analytic.
\item For $x\in\R$ the matrix $Y$ admits boundary values
  $Y_\pm(x)=\lim_{\varepsilon\to0+}Y(x\pm i\varepsilon)$, related by
  \begin{equation}\label{eq:model2jump}
    Y_+(x)=Y_-(x)\begin{pmatrix}I_r&W_n(x)\\0_r&I_r\end{pmatrix},
    \qquad x\in\R.
  \end{equation}
\item As $z\to\infty$,
  \begin{equation}\label{eq:model2infty}
    Y(z)=\bigl(I_{2r}+O(z^{-1})\bigr)z^{n\sigma_3}.
  \end{equation}
\end{enumerate}
\end{rhp}
 
No endpoint conditions are needed, the jump contour being all of $\R$, and RH
Problem \ref{rhp:model2} is solved by \eqref{eq:Ysolution} with $I=\R$ and
$W=W_n$.  The set $E$ does not appear in RH Problem \ref{rhp:model2}; it enters
only through the equilibrium measure, at the first transformation.

The two models differ in their potential theory, and their local parametrices
are different: those of the first model are built out of Bessel functions, of
orders determined by \eqref{eq:gamma} and by the orders of vanishing of the
eigenvalues of $H$ at the endpoints, as in \cite[\S3.5]{DKR23}, while those of
the second are built out of Airy functions, as in \cite[\S5.5]{DR25}.  They
agree, however, in the global problem that remains away from the endpoints.  We
show in Section \ref{subsubsec:reduction} that in both models the steepest
descent analysis leads to the following problem, with data $K_1,\dots,K_g\in\Ur$
that we identify explicitly in terms of the matrix Szeg\H{o} function of Theorem
\ref{thm:thm:WM} and of the equilibrium measure.

\begin{rhp}\label{def:rhp}
Let $K_1,\dots,K_g\in\Ur$. We seek a solution
\[
N:\C\setminus[a_0,b_{g}]\longrightarrow \C^{2r\times2r}
\]
such that:
\begin{enumerate}
\item[\textup{(N0)}] $N$ is holomorphic on $\C\setminus[a_0,b_{g}]$.
\item[\textup{(N1)}] 
\[
N_+=N_-\begin{pmatrix}0&I_r\\-I_r&0\end{pmatrix} \quad\text{on } (a_k,b_k),\ k=0,\dots,g,
\]

\item[\textup{(N2)}]$$N_+=N_-\operatorname{diag}(K_j,K_j^{-1})\quad\text{on }(b_{j-1},a_{j}),\qquad\ j=1,\dots,g;$$

\item[\textup{(N3)}] $N(z)=I_{2r}+O(z^{-1})$ as $z\to\infty$;
\item[\textup{(N4)}] as $z\to e$ with $e\in\{a_k,b_k\}$,
\[
N(z)=O\bigl(|z-e|^{-1/4}\bigr)).
\]
\end{enumerate}
\end{rhp}

We show in Section \ref{sec:gp} that the model problem is uniquely solvable. Using the model solution one then reverses the transformations $Y\mapsto\dots S\mapsto R$. This then gives strong asymptotics of $P_n$ in the entire complex plane and its recurrence coefficients. Of course sufficient care must be taken in reversing every step of the transformation and using the proper local parametrix. We refer to \cite{DKR23, DR25}  where such a calculation is done in detail.

\begin{remark}
We point out that the vector-bundle techniques developed here also apply to the scalar setting. For Jacobi-type weights, scalar Szegő functions generally acquire multiplicative jumps of modulus one across the gaps \cite[\S4.1]{KV03}. Likewise, for varying weights of the form $e^{-nV}$ \cite{DKMVZ99b}, when the equilibrium measure has multi-component support, the associated model problem has constant jumps of modulus one along cycles of the corresponding hyperelliptic Riemann surface. These multivalued scalar functions may therefore be interpreted as meromorphic sections of flat unitary line bundles.

From this perspective, the bundle-theoretic formulation allows one to bypass the usual construction in terms of quasi-periodic Jacobi theta functions and, in particular, to avoid the delicate issue of special divisors. Our main point is that, once the prescribed jump data have been encoded in a unitary bundle, the essential remaining obstruction to solving the model problem is the normalization at infinity. Proposition~\ref{prop:ev} resolves precisely this issue by showing that the relevant evaluation map is an isomorphism. Once this normalization problem has been resolved, existence and uniqueness of the model problem will be more tractable.

\end{remark}
In section \ref{vanish} we take an alternate route to show the existence and uniqueness of the global parametrix. We rely on a vanishing lemma of Proposition \ref{prop:vanishing} to show RH problem \ref{def:rhp} has a unique solution. In the next subsection \ref{subsec:fredholm} we show the associated singular integral operator has Fredholm index $0$, so trivial kernel guarantees a unique $L^2$ solution of the RH problem \ref{def:rhp}. Readers just interested in the asymptotics may entirely skip this section and move to Section \ref{subsubsec:reduction} where RH problem \ref{rhp:model1} and RH problem \ref{rhp:model2} are reduced to RH problem \ref{def:rhp}. In that setting the asymptotics of MVOPs are obtained in \eqref{eq:MVOP1} and \eqref{MVOP2}

\section{Proof of Theorem \ref{thm:thm:WM}}
\begin{proof}
    Define
\begin{equation}\label{eq:lifted-weight}
 \cM(\tau):=M(p_*(\tau))
\end{equation} Then $\cM$ satisfies all the assumptions of Theorem \ref{thm:WM}.
Thus Theorem \ref{thm:WM} and Lemma \ref{lem:canonical-normalization} gives a unique
\[
 G_0\in H^2(\D)^{r\times r}
\]
such that
\begin{equation}\label{eq:G-main}
 \cM=G_0G_0^*
 \quad\text{almost everywhere on }\T,
\end{equation}
and \(\det G_0\) is outer and $G(0)>0$.

Observe that for \(\gamma\in\Gamma\) we have by \eqref{eq:deck}
$p_*(\gamma(\tau))=p_*(\tau)$
for almost every \(\tau\).  Therefore
\begin{equation}\label{eq:cM-invariance}
\begin{aligned}
 (\cM\circ\gamma)(\tau)
=M(p_*(\gamma(\tau)))
 =M(p_*(\tau))
=\cM(\tau)
\end{aligned}
\end{equation}
almost everywhere. We then find
\begin{align*}
 (G\circ\gamma)(\tau)(G\circ\gamma)(\tau)^*
 =(GG^*)(\gamma(\tau))
 =\cM(\gamma(\tau))=\cM(\tau)
\end{align*}
Thus $G\circ\gamma$ is another factorisation (which is in $H^2(\mathbb D)^{r\times r}$ and outer), thus it must be $G\circ\gamma=GU_\gamma$, where $U_\gamma\in U(r)$.

Due to the fact $\Omega_0$ is simply connected the inclusion
\[
 \iota_0:\Omega_0\hookrightarrow\Omega
\]
has a unique lift
\begin{equation}\label{eq:section-s}
 s:\Omega_0\longrightarrow\D,
 \qquad
 p\circ s=\operatorname{id}_{\Omega_0},
 \qquad
 s(\infty)=0.
\end{equation}
Define
\[
 D(z)=G(s(z)).
\]
It is a composition of holomorphic maps, so it is holomorphic on
\(\Omega_0\).  Moreover,
\[
 \det D(z)=\det G(s(z))\neq0,
\]
so \(D(z)\) is invertible at every \(z\in\Omega_0\).  At infinity,
\[
 D(\infty)=G(s(\infty))=G(0)>0.
\]
Fix  $x\in(a_\ell,b_\ell)$. Approaching $x$ from above and from below
gives two points $s_+(x)$, $s_-(x)$ on $\T$, both  over $x$.  \eqref{eq:G-main}
therefore applies at $s_\pm(x)$ for a.e.\ $x$. Thus
\[
  D_\pm(x)D_\pm(x)^* = G(s_\pm(x))G(s_\pm(x))^* = M(x).
\]
Fix $x_0\in \Sigma_j$. Both $s_\pm(x_0)$ lie in the fibre $p^{-1}(x_0)$, on which $\Gamma$ acts
simply transitively, so there is a unique $\gamma_j\in\Gamma$ with $s_+(x_0)=\gamma_j(s_-(x_0))$.
Since $\gamma_j\circ s_-$ and $s_+$ are lifts of $\Sigma_j\hookrightarrow\Omega$ agreeing at $x_0$ and
$\Sigma_j$ is connected, unique lifting gives $s_+=\gamma_j\circ s_-$ on all of $\Sigma_j$. Hence, by the
 $G\circ\gamma_j=GU_j$ with $U_j\in U(r)$ constant,
\[
  D_+ \;=\; G\circ s_+ \;=\; G\circ\gamma_j\circ s_- \;=\; D_-\,U_j ,\qquad U_j\in\mathrm U(r).
\]

\end{proof}

\begin{lemma}\label{lem:Dunique}
Let $D$ and $\widetilde D$ both satisfy the conclusions of
Theorem~\ref{thm:thm:WM}, with unitary matrices $U_j$ and $\widetilde U_j$ as in
\eqref{eq:gap-main}. Assume $D$, $D^{-1}$, $\widetilde D$, $\widetilde D^{-1}$ are
bounded on $\Omega_0$. Then $D = \widetilde D$ and $U_j = \widetilde U_j$ for
every $j$.
\end{lemma}

\begin{proof}
Put $V = D^{-1}\widetilde D$, which is analytic and invertible on $\Omega_0$, and
note that $V$ and $V^{-1}$ are bounded there.

On a cut, \eqref{eq:band-main} holds for both functions, so almost everywhere on
$E$
\[
   V_\pm V_\pm^{\,*} = D_\pm^{-1}\widetilde D_\pm \widetilde D_\pm^{\,*}
   D_\pm^{-*} = D_\pm^{-1} M D_\pm^{-*} = I_r .
\]
Hence $V_\pm$ is unitary almost everywhere on $E$, and in particular
$\|V_\pm\| = \|V_\pm^{-1}\| = 1$ there.

On a gap, \eqref{eq:gap-main} gives $V_+ = U_j^{-1}V_-\widetilde U_j$. Let $B$ be
a disc around a point of $\Sigma_j$ meeting no cut, and set $H = V$ on the lower
half of $B$ and $H = U_jV\widetilde U_j^{-1}$ on the upper half. Then
$H_+ = U_jV_+\widetilde U_j^{-1} = V_- = H_-$, so $H$ is analytic on $B$ by
Morera's theorem. Since multiplying by unitary matrices on either side does not
change the norm, $\|H\| = \|V\|$ and $\|H^{-1}\| = \|V^{-1}\|$ on $B$. Therefore
$\|V\|$ and $\|V^{-1}\|$, which are defined on $\Omega_0$, extend continuously to
the gaps and are subharmonic on all of $\Omega$; for the subharmonicity recall
that $\log\|F\| = \sup\log|u^*Fv|$, the supremum over unit vectors $u,v$, is
subharmonic whenever $F$ is analytic.
The functions
\[
 u(z)=\log\|V(z)\|,\qquad
 \widetilde u(z)=\log\|V(z)^{-1}\|
\]
therefore extend to bounded subharmonic functions on \(\Omega\).  Their
boundary values are equal to \(0\) almost everywhere on \(E\).  Since
\(\Omega\), regarded as a domain in the Riemann sphere, contains
\(\infty\), the maximum principle for bounded subharmonic functions
gives
\[
 u(z)\leq 0,\qquad \widetilde u(z)\leq 0,\qquad z\in\Omega.
\]
Thus, for \(z\in\Omega_0\),
\[
 \|V(z)\|\leq 1,\qquad \|V(z)^{-1}\|\leq 1.
\]
For every unit vector \(v\),
\[
 1=\|V(z)^{-1}V(z)v\|\leq \|V(z)v\|\leq 1.
\]
Hence \(V(z)\) is unitary for every \(z\in\Omega_0\).
 Consequently $\sum_{j,k}|V_{jk}|^2 = \operatorname{tr}(VV^*) = r$
is constant. Each $|V_{jk}|^2$ is subharmonic and the sum is constant, so each is
harmonic, and $\Delta|V_{jk}|^2 = 4|V_{jk}'|^2$ forces $V' \equiv 0$. Thus $V$ is
a constant unitary matrix.

So $\widetilde D = DV$, and at infinity $\widetilde D(\infty) = D(\infty)V$ with
both factors positive definite, whence $V = I_r$ exactly as in the proof of
Lemma~\ref{lem:canonical-normalization}. Therefore $\widetilde D = D$, and then
\eqref{eq:gap-main} gives $\widetilde U_j = U_j$.
\end{proof}

\begin{lemma}\label{lem:Dboundary}
Let \(I\) be an open subinterval of \(E^\circ\) on which \(M\) is real
analytic and positive definite. Then \(D_\pm\) have continuous boundary
values on \(I\), and
\[
 M(x)=D_+(x)D_+(x)^*
     =D_-(x)D_-(x)^*,
 \qquad x\in I.
\]
If \(M\) is real analytic and positive definite on all of \(E\), and
\(D,D^{-1}\) are bounded on \(\Omega_0\), then \(D_\pm\) extend
continuously to the endpoints and the same identity holds for every
\(x\in E\).
\end{lemma}

\begin{proof}
Let \(A\subset\mathbb T\) be a boundary arc on which the covering map
\(p\) takes its values in \(I\). Both \(p\) and
\(\mathcal M=M\circ p\) extend holomorphically across \(A\). If \(G_0\)
is the Wiener--Masani factor, put
\[
 G_0^\sharp(w)=G_0(1/\overline w)^*
\]
and, on the exterior side of \(A\), define
\[
 \widehat G_0(w)=\mathcal M(w)
                  \bigl(G_0^\sharp(w)\bigr)^{-1}.
\]
Since \(\mathcal M=G_0G_0^*\) almost everywhere on \(A\), the boundary
values of \(G_0\) and \(\widehat G_0\) agree almost everywhere there.
They therefore extend holomorphically across \(A\). Consequently
\(D=G_0\circ s\) has continuous boundary values on \(I\), and the
almost everywhere identity extends to every \(x\in I\).

Under the assumptions in the second statement, the same argument
applies up to the preimages of the endpoints. The remaining
singularities are removable by the boundedness of \(D\) and
\(D^{-1}\).
\end{proof}
\section{Riemann surface and vector bundle preliminaries}\label{sec:he}

We denote
\begin{equation}\label{def:aobo}
A_0(z):=\prod_{\ell=0}^{g}(z-a_\ell),\qquad
B_0(z):=\prod_{\ell=0}^{g}(z-b_\ell),
\end{equation}
and let $X$ be the Riemann surface of the algebraic curve
\begin{equation}\label{eq:curve}
y^2=A_0(z)B_0(z),
\end{equation}
with hyperelliptic projection $\pi:X\to\widehat{\mathbb C}$, $\pi(z,y)=z$. The
$2g+2$ zeros of $A_0B_0$ are simple, so $X$ is hyperelliptic of genus $g$ with
simple ramification exactly over $a_0,b_0,\dots,a_{g},b_{g}$; since
$\deg(A_0B_0)=2g+2$ is even, $\infty$ is unramified and has two preimages.

 We fix the branch $y$ with $y(x)>0$
for $x>b_{g}$, and put
\[
P_1(z):=(z,y(z)),\qquad P_2(z):=(z,-y(z)),\qquad z\in\Omega .
\]
The two \emph{sheets} of $X$ are
\begin{equation}\label{eq:sheets}
X_1:=\{P_1(z):z\in\Omega\},\qquad X_2:=\{P_2(z):z\in\Omega\},
\end{equation}
so that $X=X_1\sqcup X_2\sqcup\pi^{-1}(E)$ and $\pi$ maps each sheet
biholomorphically onto $\Omega$. We write $\infty_1\in X_1$ and $\infty_2\in X_2$
for the two points over $\infty$.

Since $A_0B_0$ has real coefficients, $y(\bar z)=\overline{y(z)}$. On a cut
$A_0B_0<0$, so the boundary values $y_\pm$ are purely imaginary; counting the
$2g+1-2k$ negative factors of $A_0B_0$ on $(a_k,b_k)$ gives
\begin{equation}\label{eq:y-plus}
y_+(x)=i(-1)^{g-k}\bigl|A_0(x)B_0(x)\bigr|^{1/2}=-y_-(x),
\qquad x\in(a_k,b_k).
\end{equation}

\begin{definition}\label{def:involutions}
On $X$ define
\[
\iota(z,y):=(z,-y),\qquad \tau(z,y):=(\bar z,-\bar y).
\]
The map $\iota$ is a holomorphic involution (the sheet interchange); $\tau$ is an
\emph{antiholomorphic} involution.
\end{definition}

By \eqref{eq:sheets} and $y(\bar z)=\overline{y(z)}$ we have
$\tau(P_1(z))=P_2(\bar z)$, so $\tau$ interchanges the two sheets. Moreover
$\tau$ fixes $(z,y)$ precisely when $z\in\mathbb R$ and $y\in i\mathbb R$, i.e.\
when $A_0(z)B_0(z)\le 0$, i.e.\ when $z\in E$; hence
$\operatorname{Fix}(\tau)=\pi^{-1}(E)$ (here we denote $\operatorname{Fix}$ as the fix points). We set
\[
S:=\overline{X_1}\subset X,
\]
a compact bordered Riemann surface with $\partial S=\pi^{-1}(E)$ a disjoint union
of $g+1$ real-analytic ovals, and $X=S\cup\tau(S)$, $S\cap\tau(S)=\partial S$.

\begin{definition}\label{def:cut}
For $j=1,\dots,g$ put
\[
\Gamma_j:=\pi^{-1}\big([b_{j-1},a_{j}]\big)\subset X,\qquad X_0=X \setminus\left(\bigcup_{j=1}^{g}\Gamma_j
\right).\]
\end{definition}

Each $\Gamma_j$ is a real-analytic Jordan curve (two arcs over the open gap, joined at the
branch points $b_{j-1}$ and $a_{j}$), the $\Gamma_j$ are pairwise disjoint, and
$\Gamma_j\cap\partial S=\{b_{j-1},a_{j}\}$.

Let $W_j$ be a $\tau$-invariant open tubular neighbourhood of $\Gamma_j$ with
$\overline{W_j}\cap\overline{W_k}=\emptyset$ for $j\ne k$ and
$\overline{W_j}\cap\{\infty_{1},\infty_2\}=\emptyset$. Then $W_j\setminus\Gamma_j$ has exactly two
components $W_j^+$ and $W_j^-$, each of which is $\tau$-invariant, and,  we label them so
that
\[
W_j^+\supset\{X_1:z\ \text{near}\ \Sigma_j,\ \im z>0\}\cup\{X_2: z \hspace{0.1cm}\text{near} \hspace{0.1cm}\Sigma_j,\hspace{0.1cm}\im z<0\},\qquad
\]
and 
\[W_j^-\supset\{X_1:z \hspace{0.1cm}\text{near} \hspace{0.1cm}\Sigma_j,\hspace{0.1cm}\im z<0\}\cup\{X_2:z \hspace{0.1cm}\text{near} \hspace{0.1cm}\Sigma_j,\hspace{0.1cm}\im z>0\}
\]

We use row-vector conventions: local coordinates of a section are rows in $\C^{1\times r}$,
and a change of chart acts on the right, $s_\beta=s_\alpha T_{\alpha\beta}$, with
$T_{\alpha\beta}T_{\beta\gamma}=T_{\alpha\gamma}$.

\begin{definition}\label{def:EK}
Given $K_1,\dots,K_g\in\Ur$, let $\{X_0,W_1,\dots,W_g\}$ be the open cover of $X$ of
Definition~\ref{def:cut} and define
\[
  T_{X_0W_j}:=\begin{cases} I_r & \text{on } W_j^+,\\[2pt] K_j & \text{on } W_j^-,\end{cases}
  \qquad T_{W_jX_0}:=T_{X_0W_j}^{-1},
\]
and take all other transition matrices to be $I_r$. Since $W_j\cap W_k=\emptyset$ for $j\neq k$, the
cocycle condition holds trivially. The resulting rank-$r$ holomorphic vector bundle is
denoted $\Ek\to X$. Every transition matrix is a constant unitary matrix; thus $\Ek$ is a
flat unitary bundle and thus of degree $0$ \cite[\S6]{Gunning67}.

 We write $e^{X_0}$ and $e^{W_j}$ for the frames of $\Ek$ over $X_0$ and $W_j$
furnished by this construction. A meromorphic section is recorded in these frames
by rows of meromorphic functions, transforming on the overlap by
$U_{W_j}=U_{X_0}T_{X_0W_j}$.
\end{definition}

A change of frame acts by $v\mapsto vT$ with $T\in\Ur$ constant,
and
\[
  (wT)(vT)^*=wTT^*v^*=wv^*.
\]
Hence $wv^*$ is independent of the chart; it defines a Hermitian metric $h$ on $\Ek$ for
which the frames of Definition~\ref{def:EK} are orthonormal. Consequently
$\|s(P)\|=\bigl(s(P)s(P)^*\bigr)^{1/2}$ is well defined and continuous for every
continuous section $s$.

For a divisor $D=\sum m_PP$ on $X$ we write, as usual,
\[
H^0(X,\Ek(D)):=\{\,U\ \text{meromorphic section of }\Ek:\ \Div U+D\ge 0\,\},
\]
where $\Div U:=\sum_P(\ord  U). P$ and $\ord_P U$ is computed in any local holomorphic frame.
We write $h^0=\dim_\C H^0$.

We summarise this as follows. 
Let $U$ be a meromorphic section of $\Ek$. We can view it as a meromorphic function on $X_0$ on the frame $e^{X_0}$. Let $x\in \Sigma_j$ and let $\pm$
denote the limits from $\{\pm\im z>0\}$. Then
\begin{equation}\label{eq:vbj}
U\big(P_1(x)\big)_+=U\big(P_1(x)\big)_-K_j,
\qquad
U\big(P_2(x)\big)_+=U\big(P_2(x)\big)_-K_j^{-1},
\end{equation}
whenever the limits exist. Moreover, for $x$ in an open cut $(a_k,b_k)$,
\begin{equation}
   P_1(x)_+=P_2(x)_-\quad\text{and}\quad P_2(x)_+=P_1(x)_-\ \text{ as points of }X, 
\end{equation}

so that $U(P_1(x))_+=U(P_2(x))_-$ and $U(P_2(x))_+=U(P_1(x))_-$.

The following proposition reflects the necessity of the condition of the transition matrices being unitary.

\begin{proposition}\label{prop:pairing}
For meromorphic sections $U,V$ of $\Ek$ define,  on $X_0$,
\[
\langle U,V\rangle_\tau(P):=U(P)\,\overline{V(\tau P)}^{\mathsf T}
\qquad P\in X_0\ .
\]
Then:
\begin{enumerate}
\item $\langle U,V\rangle_\tau$ extends to a meromorphic function on all of $X$, with
\[
\Div \big(\langle U,V\rangle_\tau\big)\ \ge\ \Div U+\tau_*(\Div V),\qquad
\tau_*(V):=\textstyle\sum_P \ord_P(V)\,\tau(P);
\]
\item on $\partial S$ one has $\langle U,V\rangle_\tau=UV^*$; in particular
$\langle U,U\rangle_\tau=\|U\|^2\ge0$ there, and it vanishes at a point of $\partial S$ iff
$U$ does.
\end{enumerate}
\end{proposition}

\begin{proof}
Observe that $P\mapsto \overline{V(\tau P)}$ is holomorphic. On $X_0$ the expression is
therefore meromorphic. For the extension across $\Gamma_j$ we compute the two boundary values
on, say, the sheet 1, using \eqref{eq:vbj} for $P$
approaching from the $W_j^+$ side, $\tau P$ approaches the same side, which for sheet~2 is
the ``lower'' one; hence, with all limits taken at the same point of $\Gamma_j$,
\[
U_+=U_-K_j,\qquad V(\tau\,\cdot)_+=V(\tau\,\cdot)_-\,K_j ,
\]
the second identity because on sheet~2 the $W_j^+$ side is the $-$ side, and
$V(P_2)_-=V(P_2)_+K_j$. Therefore
\[
U_+\overline{V(\tau\cdot)_+}^{\mathsf T}
=U_-K_j\,\overline{K_j}^{\mathsf T}\,\overline{V(\tau\cdot)_-}^{\mathsf T}
=U_-K_jK_j^{*}\,\overline{V(\tau\cdot)_-}^{\mathsf T}
=U_-\overline{V(\tau\cdot)_-}^{\mathsf T},
\]
using $K_jK_j^*=I_r$. The same computation applies on the sheet-2. Thus the two
one-sided extensions agree continuously across $\Gamma_j$ minus the two branch points. The isolated branch points are removable because the function
is locally bounded there whenever $U,V$ are holomorphic there; . Statement (2) holds because $\tau|_{\partial S}=\mathrm{id}$ as it is the set of fixed points.
\end{proof}

\subsection{The required divisor and its consequences}

The following Lemma is straightforward verification of the number of sign changes, it also appears on \cite[(2.4.76)]{BL14} thus we omit the proof.
\begin{lemma}\label{lem:AB}
    Recall $A_0$ and $B_0$ in \eqref{def:aobo}. for each
$j=1,\dots,g$ there is exactly one $\xi_j\in \Sigma_j$ and $A_0(z)-B_0(z)$ has a simple zero at $\xi_j$ and these are all the $g$ zeros of $A_0(z)-B_0(z)$.
\end{lemma}

The divisor we need is the following.
\begin{definition}\label{def:D} Set
$Q_j^{+}:=\big(\xi_j,\;y(\xi_j)\big)$, $Q_j^{-}:=\big(\xi_j,\;-y(\xi_j)\big)$, and
\[
D^{+}:=\sum_{j=1}^{g}Q_j^{+},\qquad D^{-}:=\sum_{j=1}^{g}Q_j^{-},
\qquad \degree D^\pm=g .
\]
\end{definition}
We construct a meromorphic differential $\eta$, its divisor is  a direct verification.

\begin{definition}\label{def:eta}We define,
\begin{equation}\label{eq:eta}
\eta:=\frac{A_0-B_0}{y}\,dz .
\end{equation}
$\eta$ is a meromorphic $1$-form on $X$ with
\begin{equation}\label{div:eta}
\Div \eta =D^{+}+D^{-}-\infty_1-\infty_2 .
\end{equation}
In particular $\eta$ is holomorphic and nonvanishing on $E$.
\end{definition}

\begin{lemma}\label{lem:ietapos}
For every nonzero positively oriented tangent vector $v$ to $\partial S$,
\[
i\eta(v)>0 .
\]
\end{lemma}

\begin{proof}
By Definition \ref{def:eta} $v\mapsto i\eta(v)$ is a continuous, real, nowhere vanishing
function on the unit positively oriented tangent bundle of $\partial S$; each oval being
connected, its sign is constant on each oval. It therefore suffices to evaluate it at one
point of each oval.

Write $L:=\sum_{\ell=0}^{g}(b_\ell-a_\ell)$ for the total length of $E$

On $(a_k,b_k)$ Lemma~\ref{lem:AB} gives $A_0-B_0=L\prod_j(x-\xi_j)$  with $\xi_j<x$ for
$j\le k$ and $\xi_j>x$ for $j>k$, i.e.\ $g-k$ negative factors, so
$\sgn(A_0-B_0)=(-1)^{g-k}$. Also by \eqref{eq:y-plus}, 
$y_+=i(-1)^{g-k}|A_0B_0|^{1/2}$. Hence
\[
i\eta_+(x)=i\,\frac{(-1)^{g-k}|A_0(x)-B_0(x)|}{i(-1)^{g-k}|A_0(x)B_0(x)|^{1/2}}\,dx
=\frac{|A_0(x)-B_0(x)|}{|A_0(x)B_0(x)|^{1/2}}\,dx.
\]
The positive boundary orientation on an upper cut is left to right, i.e.\ $dx>0$, so $i\eta(v)>0$ there.
\end{proof}

\begin{remark}
The meromorphic function $f:=\dfrac{y-A_0}{y+A_0}$ on $X$ has divisor
\[
\Div f=D^{+}+\infty_1-D^{-}-\infty_2 .
\]

Consequently. 
\[
\Div (\eta/f)=2\big(D^--\infty_1\big),\qquad \Div(\eta f)=2\big(D^+-\infty_2\big).
\]

 Define the line bundles $\Delta_-:=\mathcal O(D^--\infty_1)$ and $\Delta_+:=\mathcal O(D^+-\infty_2)$. Both of degree
$g-1$. Then
\[
\Delta_-\otimes\Delta_+\cong \Kx,\qquad \Delta_-^{\otimes2}\cong\Kx\cong\Delta_+^{\otimes2},
\qquad \Delta_+\cong\Delta_- .
\]
\end{remark}

The following proposition is essential for the construction of the global parametrix as we shall see.
\begin{theorem}\label{thm:vanishing}
For every $K_1,\dots,K_g\in\Ur$,
\[
H^0\big(X,\Ek(D^--\infty_1)\big)=0
\qquad\text{and}\qquad
H^0\big(X,\Ek(D^+-\infty_2)\big)=0 .
\]
Equivalently: a meromorphic section $U$ of $\Ek$ whose only poles are at most simple poles at
the points $Q_1^-,\dots,Q_g^-$ and which vanishes at $\infty_1$ is identically zero; and likewise
with $Q_j^-\leftrightarrow Q_j^+$, $\infty_1\leftrightarrow \infty_2$.
\end{theorem}

\begin{proof}
Let $U\in H^0(X,\Ek(D^--\infty_1))$, i.e.\ $(U)\ge \infty_1-D^-$. Consider the $1$-form
\begin{equation}\label{eq:omega}
\omega:=\langle U,U\rangle_\tau\,\eta
=U(P)\,\overline{U(\tau P)}^{\mathsf T}\,\eta(P).
\end{equation}

  Proposition~\ref{prop:pairing} gives $\langle U,U\rangle_\tau$ is meromorphic. Using the divisor of $\eta$ from \eqref{div:eta} we find
\[
\Div \omega \ \ge\ \Div U+\tau_*\Div(U)+\Div \eta
\ \ge\ \big(\infty_1-D^-\big)+\big(\infty_2-D^+\big)+\big(D^++D^--\infty_1-\infty_2\big)=0 .
\]
Thus $\omega\in H^0(X,\Kx)$. 

 A holomorphic $1$-form is closed, thus $d\omega=0$, and $\omega$ is smooth
on the compact bordered surface $S$. Hence
\begin{equation}
0=i\int_S d\omega=i\int_{\partial S}\omega=\int_{\partial S}\|U\|^2\,i\eta ,
\end{equation}
using Proposition~\ref{prop:pairing}(2) on $\partial S$, where $\tau P=P$.

 The integrand is continuous on $\partial S$ ($U$ has no poles
there, and $\|U\|^2$ is well defined and nonnegative), and
$i\eta>0$ by Lemma~\ref{lem:ietapos}. Hence $U\equiv0$ on $\partial S$.  The identity theorem gives $U\equiv0$ on $X$.

For the second space $H^0\big(X,\Ek(D^+-\infty_2)\big)$,  we repeat the same argument with $D^-$ and $\infty_1$ replaced by
$D^+$ and $\infty_2$.
\end{proof}
We will need the Riemann--Roch theorem for holomorphic vector bundles in our setting. We recall the notions again for convenience of reading. For a holomorphic vector bundle $V$ over $X$, we denote by $H^0(X,V)$ the vector space of its global holomorphic sections and write
\[
h^0(X,V):=\dim H^0(X,V).
\]
When the underlying surface $X$ is clear from the context, we simply write $h^0(V)$.
\begin{theorem}[Riemann--Roch, {\cite[\S4]{Gunning67}}]\label{thm:RR}
Let $V$ be a holomorphic vector bundle of rank $r$ and degree $d$ on a compact Riemann
surface of genus $g$. Then $h^0(V)\geq d+r(1-g)$.
\end{theorem}

Since $\infty_1,\infty_2\in X_0$ and $D_\pm$ is supported on $\bigcup_j\Gamma_j$, a
section of $V_\pm$ is holomorphic at both points, and we may evaluate it there. We write
\[
\ev_{\infty_1}(U):=U(\infty_1)\in\C^{1\times r},
\]
the components being taken in the frame $e^{X_0}$ of Definition~\ref{def:EK}, and
similarly for $\ev_{\infty_2}$. These are linear maps, and the frame is the same one in
which the normalisations of RH Problem~\ref{def:rhpF} are expressed.
\begin{proposition}\label{prop:ev}
Let $V_-:=\Ek(D^-)$ and $V_+:=\Ek(D^+)$. Then
\[
\ev_{\infty_1}:H^0(X,V_-)\xrightarrow{\ \sim\ }\C^{1\times r},
\qquad
\ev_{\infty_2}:H^0(X,V_+)\xrightarrow{\ \sim\ }\C^{1\times r}
\]
are isomorphisms of complex vector spaces. In particular $h^0(V_\pm)=r$.
\end{proposition}

\begin{proof}
$\ker \ev_{\infty_1}=H^0(X,\Ek(D^--\infty_1))=0$ by Theorem~\ref{thm:vanishing}, so $\ev_{\infty_1}$ is
injective and $h^0(V_-)\le r$. On the other hand $\degree V_-=\degree \Ek+r\degree D^-=rg$  so Theorem~\ref{thm:RR} gives
\[
h^0(V_-)\geq rg+r(1-g)=r
\]
Hence $h^0(V_-)=r$, and an injective linear map between spaces of equal finite
dimension is bijective. The case of $V_+$ is identical.
\end{proof}

\section{The global parametrix}\label{sec:gp}
In this section we show how to construct the solution to model problem \ref{def:rhp} with our previously defined notions of vector bundles.
\subsection{Ansatz}
\begin{equation}\label{eq:beta}
\beta(z):=\prod_{\nu=0}^{g}\Bigl(\frac{z-b_\nu}{z-a_\nu}\Bigr)^{1/4},
\end{equation}
each factor being the principal fourth root; this is well defined because
$z\mapsto(z-b_\nu)/(z-a_\nu)$ maps $\Ch\setminus[a_\nu,b_\nu]$ onto
$\C\setminus(-\infty,0]$ and sends $\infty$ to $1$. Thus $\beta$ is holomorphic and
nonvanishing on $\Ch\setminus E$, with $\beta(\infty)=1$; in particular $\beta$ is
holomorphic across the gaps. Put
\begin{equation}\label{eq:cs}
\cc:=\frac{\beta+\beta^{-1}}{2},\qquad \ss:=\frac{\beta-\beta^{-1}}{2i},
\end{equation}
so that $\cc\pm i\ss=\beta^{\pm1}$ and $\cc^{2}+\ss^{2}=1$.

These functions, and the jump relations of Lemma~\ref{lem:beta}(ii), are classical;
see \cite[\S2.4]{BL14}.

\begin{lemma}\label{lem:beta}
On $\Ch\setminus E$ one has:
\begin{enumerate}[label=(\roman*)]
\item $\beta^{4}=B_0/A_0$ and $\beta^{2}=y/A_0$;
\item $\beta_+=i\beta_-$ on each cut $(a_k,b_k)$, hence there
$\cc_+=-\ss_-$ and $\ss_+=\cc_-$;
\item $\cc$ is nonvanishing, while $\ss$ vanishes exactly at $\xi_1,\dots,\xi_g$ and
at $\infty$, each zero being simple;
\item as $z\to\infty$,
\[
\beta=1-\frac{L}{4z}+O(z^{-2}),\qquad \cc=1+O(z^{-2}),\qquad
\ss=\frac{iL}{4z}+O(z^{-2}).
\]
\end{enumerate}
\end{lemma}
By abuse of notation we denote $\cc(z)=\cc(z)I_r$ and $\ss(z)=\ss(z)I_r$ the scalar matrices.

\begin{definition}\label{def:N}
For $z\in\C\setminus[a_0,b_{g}]$ put
\begin{equation}\label{eq:N}
N(z):=\begin{pmatrix}\cc(z)F_1(z) & \ss(z)F_2(z)\\[2pt]
-\ss(z)G_1(z) & \cc(z)G_2(z)\end{pmatrix},
\end{equation}
where the entries are $r\times r$ blocks.

\end{definition}

Using Lemma \ref{lem:beta} if $N$ satisfies RH problem \ref{def:rhp} then we can write the RH problem for $F_1(z), F_2(z)$ and $G_1(z), G_2(z)$.
We will see we can not solve $F$, $G$ problem without poles. However using Lemma \ref{lem:beta}(iii) we see $\ss$ has zero at $\xi_j$. Appealing to the ansatz \eqref{eq:N} we can then allow simple poles of $F_2$ and $G_1$ at $\xi_j$, and in the next proposition we will see this then guarantees a solution.

\subsection{Solution using sections of the vector bundle \texorpdfstring{$E_K(D^\pm)$}{EK(D±)}}
\begin{rhp}\label{def:rhpF}
Let $K_1,\dots,K_g\in\Ur$. We seek $H_1,H_2:\Omega_0\to\C^{r\times r}$ such that:
\begin{enumerate}
\item[(H0)] $H_1,H_2$ are holomorphic on $\Omega_0=\Ch\setminus[a_0,b_{g}]$;
\item[(H1)] on each cut $(a_k,b_k)$, $k=0,\dots,g$,
\[
H_{1+}=H_{2-},\qquad H_{2+}=H_{1-};
\]
\item[(H2)] on each gap $(b_{j-1},a_{j})$, $j=1,\dots,g$,
\[
H_{1+}=H_{1-}K_j,\qquad H_{2+}=H_{2-}K_j^{-1};
\]
\item[(H3)] $H_1,H_2$ are bounded near each branch point $a_k,b_k$.
\end{enumerate}
We impose one of two normalisations and each normalisation with a corresponding solution:
\begin{enumerate}
\item[(F)] $H_1(\infty)=I_r$, $H_2$ bounded at $\infty$; $H_1$ bounded at each $\xi_j$
and $H_2(z)=O\bigl(|z-\xi_j|^{-1}\bigr)$ as $z\to\xi_j$. A solution is written
$(F_1,F_2)$.
\item[(G)] $H_2(\infty)=I_r$, $H_1$ bounded at $\infty$; $H_2$ bounded at each $\xi_j$
and $H_1(z)=O\bigl(|z-\xi_j|^{-1}\bigr)$ as $z\to\xi_j$. A solution is written
$(G_1,G_2)$.
\end{enumerate}
\end{rhp}

\begin{lemma}\label{lem:FGsections}
RH problem~\ref{def:rhpF} with normalisation (F) has a unique solution, and likewise with
(G).
\end{lemma}

\begin{proof}
We treat the normalisation (F); the case (G) is identical, with the two sheets, the
divisors $D_\mp$, and the points $\infty_{1,2}$ interchanged throughout. Sections are
always written in the frames of Definition~\ref{def:EK}.

 Let $U\in H^0\bigl(X,\Ek(D_-)\bigr)$. As $D_-$ is
supported on $\bigcup_j\Gamma_j$, the section $U$ is holomorphic on $X_0$, and we may set
\begin{equation}\label{eq:H-from-U}
H_1(z):=U\bigl(P_1(z)\bigr),\qquad H_2(z):=U\bigl(P_2(z)\bigr),\qquad z\in\Omega_0 .
\end{equation}
Condition (H0) is then immediate. For (H1), let $x\in(a_k,b_k)$.
 The boundary values $P_1(x)_+$ and $P_2(x)_-$ are the same point
of $X$, as are $P_2(x)_+$ and $P_1(x)_-$; both lie in $X_0$, where $U$ is single-valued.
Hence
\[
H_{1+}(x)=U\bigl(P_1(x)_+\bigr)=U\bigl(P_2(x)_-\bigr)=H_{2-}(x),
\]
and likewise $H_{2+}=H_{1-}$. 

Condition (H2) is exactly
\eqref{eq:vbj}.

For (H3), a branch point $e\in\{a_k,b_k\}$, $U_{W_j}$ is holomorphic in a local frame with  $w=(z-e)^{1/2}$
 a local coordinate there, so $U$ is holomorphic, in particular bounded, at $e$.

It remains to study the poles. Near $\Gamma_j$ we pass to the frame $e^{W_j}$, in which
the components are $U_{W_j}=U_{X_0}T_{X_0W_j}$. By hypothesis $U_{W_j}$ is holomorphic at
$Q_j^+=P_1(\xi_j)$ and has at most a simple pole at $Q_j^-=P_2(\xi_j)$. Since
$T_{X_0W_j}$ is unitary and constant, $H_1$ stays bounded as $z\to\xi_j$ while
$H_2(z)=O\bigl(|z-\xi_j|^{-1}\bigr)$, which is the blow up prescribed in (F).

Conversely, let $(H_1,H_2)$ be a row which satisfy (H0)--(H3) and the blow up conditions of (F),
and defining $U$ on $\pi^{-1}(\Omega_0)$ by \eqref{eq:H-from-U} we find that $U\in H^0\bigl(X,\Ek(D_-)\bigr)$.

 A solution $(F_1,F_2)$ is an $r\times 2r$ matrix, that is, a choice of
$r$ sections $U^{(1)},\dots,U^{(r)}\in H^0\bigl(X,\Ek(D_-)\bigr)$ taken as the rows of
$\bigl(F_1\ \ F_2\bigr)$ and stacked in columns. The normalisation $F_1(\infty)=I_r$ reads
\[
\ev_{\infty_1}\bigl(U^{(i)}\bigr)=\varepsilon_i,\qquad i=1,\dots,r,
\]
$\varepsilon_i$ being the $i$-th standard row vector. By Proposition~\ref{prop:ev} the
map $\ev_{\infty_1}$ is an isomorphism onto $\C^{1\times r}$, so each $U^{(i)}$ exists and
is unique. This proves the lemma.
\end{proof}

\begin{proposition}\label{prop:Nsolves}
With $F_k,G_k$ as in Lemma~\ref{lem:FGsections}, the matrix $N$ of
Definition~\ref{def:N} solves RH Problem~\ref{def:rhp}.
\end{proposition}

\begin{proof}
(N0) holds by (H0) together with the boundedness statements at the $\xi_j$, since
$\ss$ has a simple zero there and $\cc$ is nonvanishing.

For (N1), Lemma~\ref{lem:beta}(ii) gives $\cc_+=-\ss_-$ and $\ss_+=\cc_-$ on
$(a_k,b_k)$, so by (H1)
\[
N_+=\begin{pmatrix}-\ss_-F_{2-} & \cc_-F_{1-}\\ -\cc_-G_{2-} & -\ss_-G_{1-}\end{pmatrix}
=N_-\begin{pmatrix}0&I_r\\-I_r&0\end{pmatrix}.
\]

For (N2), $\cc$ and $\ss$ are holomorphic across the gaps, so (H2) gives
$N_+=N_-\diag(K_j,K_j^{-1})$ directly.

For (N3), Lemma~\ref{lem:beta}(iv) gives $\cc(\infty)=1$ and $\ss(\infty)=0$; with
$F_1(\infty)=G_2(\infty)=I_r$ and $F_2,G_1$ bounded at $\infty$ this yields
$N(\infty)=I_{2r}$, and the error is $O(z^{-1})$ since $\ss=O(z^{-1})$ and
$\cc=1+O(z^{-2})$.

For (N4), $\beta^{\pm1}=O\bigl(|z-e|^{-1/4}\bigr)$ at each endpoint $e$, hence so are
$\cc$ and $\ss$, while $F_k,G_k$ are bounded there by (H3).
\end{proof}

We now turn to uniqueness of the model problem.
Throughout put
\begin{equation}\label{eq:CLambda}
J:=\begin{pmatrix}0&I_r\\-I_r&0\end{pmatrix},\quad
\delta:=\begin{pmatrix}iI_r&0\\0&-iI_r\end{pmatrix},\quad
\mathcal C:=\begin{pmatrix}I_r&I_r\\iI_r&-iI_r\end{pmatrix},\quad
\Lambda:=\begin{pmatrix}\beta I_r&0\\0&\beta^{-1}I_r\end{pmatrix}.
\end{equation}
Observe, $J\mathcal C=\mathcal C\delta$, as one checks
by multiplying out; and $\det J=1$, $\det\mathcal C=(-2i)^r$, $\det\Lambda=1$. Next, by
Lemma~\ref{lem:beta}(ii) one has $\beta_+=i\beta_-$ on $E$, hence
\begin{equation}\label{eq:Lambda-jump}
\Lambda_+=\Lambda_-\delta \quad E
\end{equation}
while $\Lambda$ is holomorphic across every gap and on $\R\setminus[a_0,b_{g}]$. Third,
near an endpoint $e\in\{a_j,b_j\}_{j=0}^g$ exactly one of $\beta,\beta^{-1}$ is unbounded, and it grows like
$|z-e|^{-1/4}$; therefore
\begin{equation}\label{eq:Lambda-growth}
\Lambda(z)^{\pm1}=O\bigl(|z-e|^{-1/4}\bigr),\qquad z\to e .
\end{equation}

\begin{lemma}\label{lem:merge}
Let $N$ satisfy \textup{(N0)--(N2)} and let $e$ be an endpoint. Denote by $\Sigma_{j}$ the
gap adjacent to $e$ and put $K:=K_{j}$, with the convention $K:=I_r$ when $e=a_0$ or
$e=b_{g}$. Set
\[
\Delta_e:=\begin{cases}\diag(I_r,K), & \im z>0,\\[2pt] \diag(K,I_r), & \im z<0.\end{cases}
\]
Then $\hat N:=N\Delta_e$ has no jump on $\Sigma_j$, and $\hat N_+=\hat N_-J$ on the cut
adjacent to $e$.
\end{lemma}

\begin{proof}
On $\Sigma_j$ we have $N_+=N_-\diag(K,K^{-1})$, so
\[
\hat N_+=N_+\diag(I_r,K)=N_-\diag(K,K^{-1})\diag(I_r,K)=N_-\diag(K,I_r)=\hat N_-.
\]
On the adjacent cut $N_+=N_-J$, so
\[
\hat N_+=\hat N_-\,\diag(K,I_r)^{-1}J\,\diag(I_r,K)
=\hat N_-\begin{pmatrix}0&K^{-1}K\\-I_r&0\end{pmatrix}=\hat N_-J. \qedhere
\]
\end{proof}

\begin{lemma}\label{lem:Phi}
Let $N$ satisfy \textup{(N0)--(N4)} and let $e$ be an endpoint. Then
\begin{equation}\label{eq:Phi-def}
\Phi:=\hat N\,\mathcal C\,\Lambda^{-1}
\end{equation}
extends holomorphically to a full disc centred at $e$.
\end{lemma}

\begin{proof}
Work in a disc $D$ about $e$ so small that it meets no other endpoint. By
Lemma~\ref{lem:merge}, $\hat N$ continues analytically across $D\cap\Sigma_j$, and
$\Lambda$ does too; so $\Phi$ is analytic on $D$ minus the cut adjacent to $e$. Across
that cut, using $\hat N_+=\hat N_-J$, then $J\mathcal C=\mathcal C\delta$, then
\eqref{eq:Lambda-jump},
\[
\Phi_+=\hat N_-J\mathcal C\Lambda_+^{-1}
      =\hat N_-\mathcal C\,\delta\,\bigl(\Lambda_-\delta\bigr)^{-1}
      =\hat N_-\mathcal C\Lambda_-^{-1}=\Phi_-.
\]
 Hence $\Phi$ is analytic on $D\setminus\{e\}$. By (N4) and
\eqref{eq:Lambda-growth},
\[
\Phi(z)=O\bigl(|z-e|^{-1/4}\bigr)\,O\bigl(|z-e|^{-1/4}\bigr)=O\bigl(|z-e|^{-1/2}\bigr),
\]
so the isolated singularity at $e$ is removable.
\end{proof}

\begin{proposition}\label{prop:detN}
Every solution $N$ of RH Problem~\ref{def:rhp} satisfies $\det N\equiv1$. In particular
$N(z)$ is invertible for all $z\in\C\setminus[a_0,b_{g}]$, and
$N^{-1}(z)=O\bigl(|z-e|^{-1/4}\bigr)$ as $z\to e$ at each endpoint.
\end{proposition}

\begin{proof}
Both jump matrices are unimodular, since $\det J=1$ and
$\det\diag(K_j,K_j^{-1})=1$. Hence $\det N$ has no jump on $(a_0,b_{g})$ and extends
analytically to $\C$ minus the $2g+2$ endpoints. Near an endpoint $e$,
\eqref{eq:Phi-def} gives
\[
\det N=\bigl(\det\Delta_e\bigr)^{-1}\det\hat N
      =\bigl(\det K\bigr)^{-1}(-2i)^{-r}\det\Phi ,
\]
which is bounded by Lemma~\ref{lem:Phi}. So all the isolated singularities are removable,
$\det N$ is entire, and $\det N\to1$ at infinity by (N3); Liouville's theorem gives
$\det N\equiv1$.

Consequently $\det\Phi=(-2i)^r\det K$ is a nonzero constant near $e$, so $\Phi(e)$ is
invertible and $\Phi^{-1}$ is holomorphic at $e$, in particular bounded. Inverting
\eqref{eq:Phi-def} and using \eqref{eq:Lambda-growth},
\[
N^{-1}=\Delta_e\,\mathcal C\,\Lambda^{-1}\Phi^{-1}=O\bigl(|z-e|^{-1/4}\bigr). \qedhere
\]
\end{proof}

\begin{proposition}\label{prop:unique}
RH Problem~\ref{def:rhp} has exactly one solution.
\end{proposition}

\begin{proof}
Existence is Proposition~\ref{prop:Nsolves}. Let $\widetilde N$ be another solution and
put $H:=\widetilde NN^{-1}$, well defined and analytic on $\C\setminus[a_0,b_{g}]$ by
Proposition~\ref{prop:detN}. As $\widetilde N$ and $N$ have the same jump matrix $J_N$ on
each cut and each gap,
\[
H_+=\widetilde N_-J_N\,J_N^{-1}N_-^{-1}=H_- ,
\]
so $H$ extends analytically across $(a_0,b_{g})$ off the endpoints. There (N4) for
$\widetilde N$ and Proposition~\ref{prop:detN} for $N^{-1}$ give
$H=O(|z-e|^{-1/2})$, so the singularities are removable and $H$ is entire. By (N3),
$H\to I_{2r}$ at infinity, so $H\equiv I_{2r}$ by Liouville's theorem.
\end{proof}

\section{Vanishing lemma and Fredholm index}\label{vanish}
\subsection{A vanishing lemma}
\label{subsec:vanishing}

Proposition \ref{prop:unique} proves uniqueness for RH Problem \ref{def:rhp}  using that $\det N\equiv1$.  We give an alternate existence and uniqueness proof of the model problem. The first step is a vanishing lemma.  The argument is the  contour integral vanishing
lemma of Riemann--Hilbert theory, going back to Zhou \cite{Zho89}; see also
\cite[Ch.~7]{Dei99}.

The usual hypothesis in such arguments is that the jump matrix $v$ has positive
semi-definite Hermitian part $v+v^*$ on the whole contour, with strict positivity
on a set of positive measure.  
We set for ease of notations,
\begin{equation}\label{eq:vanishing-jump}
  v(x):=
  \begin{cases}
    J,&x\in(a_\ell,b_\ell),\ \ell=0,\dots,g,\\[2pt]
    \operatorname{diag}\bigl(K_j,K_j^{-1}\bigr),&x\in\Sigma_j,\ j=1,\dots,g,\\[2pt]
    I_{2r},&x\in\R\setminus[a_0,b_g],
  \end{cases}
\end{equation}
so that (N1) and (N2) gives $\widehat N_+=\widehat N_-v$ on $\R\setminus\mathcal E$.

Recall $J=\begin{psmallmatrix}0_r&I_r\\-I_r&0_r\end{psmallmatrix}$ from
\eqref{eq:CLambda}.  Then $J^2=-I_{2r}$ and $J^*=-J$, so that $J^{-1}=J^*=-J$;
in particular $J$ is unitary.  We set
\begin{equation}\label{eq:defA}
  \mathcal A:=J^{-1}=-J=
  \begin{pmatrix}0_r&-I_r\\I_r&0_r\end{pmatrix}.
\end{equation}
We record  some observations in the following lemma.
\begin{lemma}\label{lem:algebraic}
Let $\mathcal A$ be as in \eqref{eq:defA} and let $K\in\C^{r\times r}$ be
invertible.  Then
\begin{enumerate}
\item[(i)] $\mathcal A^*=-\mathcal A$, and consequently
  $I_{2r}\mathcal A+(I_{2r}\mathcal A)^*=0_{2r}$;
\item[(ii)] $J\mathcal A=I_{2r}$, and consequently
  $J\mathcal A+(J\mathcal A)^*=2I_{2r}$, which is positive definite;
\item[(iii)]Let $K$ be unitary, writing $D:=\operatorname{diag}(K,K^{-1})$, one has
  \[
    D\mathcal A+(D\mathcal A)^*=0_{2r}
    \qquad\text{for}\qquad K\in\Ur .
  \]
\end{enumerate}
\end{lemma}

\begin{proof}
(i) and (ii) follows from \eqref{eq:defA}. 

(iii) A direct computation gives
\[
  D\mathcal A
  =\begin{pmatrix}K&0_r\\0_r&K^{-1}\end{pmatrix}
   \begin{pmatrix}0_r&-I_r\\I_r&0_r\end{pmatrix}
  =\begin{pmatrix}0_r&-K\\K^{-1}&0_r\end{pmatrix},
  \qquad
  (D\mathcal A)^*
  =\begin{pmatrix}0_r&(K^{-1})^*\\-K^*&0_r\end{pmatrix}.
\]
Hence $(D\mathcal A)^*=-D\mathcal A$ holds if and only if $(K^{-1})^*=K$ and
$K^*=K^{-1}$.
\end{proof}

Observe that $K$ being unitary is crucial in (iii) of Lemma \ref{lem:algebraic}
This is the same spirit as in Proposition \ref{prop:pairing}, where
unitarity is what makes the pairing $\langle U,V\rangle_\tau$ extend across
$\Gamma_j$.

\begin{proposition}\label{prop:vanishing}
Let $K_1,\dots,K_g\in\Ur$ and let
$\widehat N:\C\setminus[a_0,b_g]\to\C^{2r\times2r}$ satisfy the homogeneous
version of RH Problem \textup{\ref{def:rhp}}, namely
\begin{enumerate}
\item[\textup{($\widehat{\text N}$0)}] $\widehat N$ is holomorphic on
  $\C\setminus[a_0,b_g]$;
  $\widehat N_\pm$ on $\R\setminus\mathcal E$;
\item[\textup{($\widehat{\text N}$1)}] $\widehat N_+=\widehat N_-v$ on
  $\R\setminus\mathcal E$, with $v$ as in \eqref{eq:vanishing-jump};
\item[\textup{($\widehat{\text N}$2)}] $\widehat N(z)=O(z^{-1})$ as
  $z\to\infty$;
\item[\textup{($\widehat{\text N}$3)}]
  $\widehat N(z)=O\bigl(|z-e|^{-1/4}\bigr)$ as $z\to e$, for each
  $e\in\mathcal E$.
\end{enumerate}
Then $\widehat N\equiv0$.
\end{proposition}

\begin{proof}
Write $\C_\pm=\{z\in\C:\pm\im z>0\}$ and define
\begin{equation}\label{eq:defQ}
  Q(z):=\widehat N(z)\,\mathcal A\,\widehat N(\bar z)^*,
  \qquad z\in\C_+ .
\end{equation}

Since
$z\mapsto\widehat N(\bar z)^*$ is holomorphic on $\C_+$ so is $Q$, and for $x\in\R\setminus\mathcal E$,
\begin{equation}\label{eq:Qplus}
  Q_+(x)=\widehat N_+(x)\,\mathcal A\,\widehat N_-(x)^*,
\end{equation}
because $z\to x$ in $\C_+$ means $\bar z\to x$ in $\C_-$.

By ($\widehat{\text N}$2), $\widehat N(z)=O(z^{-1})$ and
$\widehat N(\bar z)^*=O(z^{-1})$ as $z\to\infty$ in $\C_+$, so
\begin{equation}\label{eq:Qinfty}
  Q(z)=O(z^{-2}),\qquad z\to\infty .
\end{equation}
By ($\widehat{\text N}$3), for $e\in\mathcal E$,
\begin{equation}\label{eq:Qendpoint}
  Q(z)=O\bigl(|z-e|^{-1/2}\bigr),\qquad z\to e .
\end{equation}
In particular $Q_+\in L^1(\R)$.  We next show, $\int_\R Q_+(x)\,dx=0$.
Let $0<\varepsilon<\tfrac12\min\{|e-e'|:e\ne e'\in\mathcal E\}$ and let
$R>\max_{e\in\mathcal E}|e|+1$.  Put
\[
  \Omega_{R,\varepsilon}
  :=\bigl\{z\in\C_+:|z|<R \text{ and } |z-e|>\varepsilon
    \text{ for all }e\in\mathcal E\bigr\}.
\]
Then $Q$ is holomorphic on $\Omega_{R,\varepsilon}$ and continuous on
$\overline{\Omega_{R,\varepsilon}}$, so Cauchy's theorem gives
\begin{equation}\label{eq:cauchy}
  \int_{\partial\Omega_{R,\varepsilon}}Q(z)\,dz=0 ,
\end{equation}
the boundary being positively oriented: the parts of $[-R,R]$ at distance at
least $\varepsilon$ from $\mathcal E$ traversed from left to right, the upper
semicircles $|z-e|=\varepsilon$ traversed clockwise, and the semicircle $|z|=R$
traversed counterclockwise.

On the large semicircle, \eqref{eq:Qinfty} gives $\|Q\|=O(R^{-2})$ while the
length is $\pi R$, so that contribution is $O(R^{-1})$ and tends to $0$ as
$R\to\infty$; by \eqref{eq:Qinfty} the integral over the real part converges
absolutely, so letting $R\to\infty$ in \eqref{eq:cauchy} with $\varepsilon$
fixed yields
\[
  \int_{\{x\in\R:\,|x-e|\ge\varepsilon\ \forall e\in\mathcal E\}}Q_+(x)\,dx
  \;+\;\sum_{e\in\mathcal E}\int_{|z-e|=\varepsilon,\ \im z>0}Q(z)\,dz=0 .
\]
On a small semicircle, \eqref{eq:Qendpoint} gives
$\|Q\|=O(\varepsilon^{-1/2})$ while the length is $\pi\varepsilon$, so that
contribution is $O(\varepsilon^{1/2})$ and tends to $0$ as
$\varepsilon\downarrow0$.  Since $Q_+\in L^1(\R)$, dominated convergence gives
\begin{equation}\label{eq:Qzero}
  \int_\R Q_+(x)\,dx=0 .
\end{equation}

By \eqref{eq:Qplus} and ($\widehat{\text N}$1),
$Q_+=\widehat N_-\,v\mathcal A\,\widehat N_-^*$ on $\R\setminus\mathcal E$, so
\eqref{eq:Qzero} reads
\[
  \int_\R\widehat N_-(x)\,v(x)\mathcal A\,\widehat N_-(x)^*\,dx=0_{2r}.
\]
Taking Hermitian adjoints and adding the result to this identity gives
\begin{equation}\label{eq:hermitianpart}
  \int_\R\widehat N_-(x)
    \Bigl(v(x)\mathcal A+\bigl(v(x)\mathcal A\bigr)^*\Bigr)
    \widehat N_-(x)^*\,dx=0_{2r}.
\end{equation}
By Lemma \ref{lem:algebraic},
\[
  v\mathcal A+(v\mathcal A)^*
  =\begin{cases}
     2I_{2r},&\text{on the cuts, by (ii)},\\
     0_{2r},&\text{on the gaps, by (iii), since }K_j\in\Ur,\\
     0_{2r},&\text{on }\R\setminus[a_0,b_g],\text{ by (i)} .
   \end{cases}
\]
Therefore \eqref{eq:hermitianpart} becomes
\[
  2\int_{E^\circ}\widehat N_-(x)\widehat N_-(x)^*\,dx=0_{2r},
\]
 Taking traces and using
$\operatorname{tr}(AA^*)=\sum_{p,q}|A_{pq}|^2=\|A\|_{\mathrm{HS}}^2$, we obtain
\[
  \int_{E^\circ}\bigl\|\widehat N_-(x)\bigr\|_{\mathrm{HS}}^2\,dx=0 ,
\]
the integral being finite because $\|\widehat N_-\|^2_{\mathrm{HS}}
=O(|x-e|^{-1/2})$ at each endpoint by ($\widehat{\text N}$3).  The integrand is
nonnegative, so
\begin{equation}\label{eq:Nminuszero}
  \widehat N_-(x)=0_{2r}\qquad\text{for almost every }x\in E^\circ,
\end{equation}
and then $\widehat N_+=\widehat N_-J=0_{2r}$ almost everywhere on $E^\circ$ as
well.
Fix a cut $(a_\ell,b_\ell)$ and define
\[
  \widetilde N(z):=
  \begin{cases}
    \widehat N(z),&\im z<0,\\
    \widehat N(z)J^{-1},&\im z>0 .
  \end{cases}
\]
For $x\in(a_\ell,b_\ell)$ the boundary values satisfy
$\widetilde N_+(x)=\widehat N_+(x)J^{-1}=\widehat N_-(x)JJ^{-1}
=\widehat N_-(x)=\widetilde N_-(x)$.  By Morera's theorem $\widetilde N$ extends
holomorphically across $(a_\ell,b_\ell)$; denote by $\Lambda$ the open set
$\C_+\cup\C_-\cup(a_\ell,b_\ell)$, which is connected, and by $\widetilde N$ the
extension, holomorphic on $\Lambda$.

The restriction of $\widetilde N$ to $(a_\ell,b_\ell)$ equals $\widehat N_-$,
hence vanishes almost everywhere there by \eqref{eq:Nminuszero}; being
continuous, it vanishes identically on $(a_\ell,b_\ell)$ so the identity theorem gives $\widetilde N\equiv0$ on
$\Lambda$.  In particular $\widehat N\equiv0$ on $\C_+\cup\C_-$, and by
continuity on all of $\C\setminus[a_0,b_g]$.
\end{proof}

\begin{corollary}\label{cor:vanishing-uniqueness}
RH Problem \ref{def:rhp} has at most one solution.
\end{corollary}

\begin{proof}
If $N^{(1)}$ and $N^{(2)}$ are solutions, put $\widehat N:=N^{(1)}-N^{(2)}$.
Then $\widehat N$ is holomorphic on $\C\setminus[a_0,b_g]$ and, the jump
relations being linear in $N$, satisfies $\widehat N_+=\widehat N_-v$.  By (N3)
we have $N^{(i)}=I_{2r}+O(z^{-1})$, so the constant terms cancel and
$\widehat N(z)=O(z^{-1})$, and by (N4) we have
$\widehat N(z)=O(|z-e|^{-1/4})$.  Proposition \ref{prop:vanishing} gives
$\widehat N\equiv0$.
\end{proof}

\subsection{Solvability by Fredholm theory}
\label{subsec:fredholm}

In this subsection we prove unique solvability of RH Problem \ref{def:rhp}
without using the bundle $E_K$.  The vanishing lemma of Section
\ref{subsec:vanishing} gives injectivity of the associated singular integral
operator; what has to be added is that this operator has Fredholm index zero,
and this is the content of Proposition \ref{prop:indexzero}.  The reader who is
content with Propositions \ref{prop:Nsolves} and \ref{prop:unique} may skip the
subsection; we include it because the index computation turns out to be
completely explicit.
\subsubsection{The \texorpdfstring{$L^2$}{L2} formulation}

Let $\Gamma:=\widehat\R=\R\cup\{\infty\}$, and observe $v$ is piecewise constant, its only
discontinuities on $\Gamma$ are the $2g+2$ points of
$\mathcal E$, and $v^{\pm1}\in L^\infty(\Gamma)$ with
$\|v\|_\infty=\|v^{-1}\|_\infty=1$.

Write $\mathcal C$ for the Cauchy transform,
\[
  (\mathcal C\phi)(z)=\frac1{2\pi i}\int_\Gamma\frac{\phi(s)}{s-z}\,ds,
  \qquad z\in\C\setminus\R,
\]
acting entrywise, and $\mathcal C_\pm$ for its non-tangential boundary values.
By the Sokhotskii-Plemelj formula $\mathcal C_+-\mathcal C_-=\mathrm{id}$ and
$\mathcal C_++\mathcal C_-=S$, where $S$ is the Cauchy principal value integral; since
$S^2=\mathrm{id}$ on $L^2(\R)$, the operators
\begin{equation}\label{eq:projections}
  P:=\mathcal C_+=\tfrac12(\mathrm{id}+S),
  \qquad
  Q:=-\mathcal C_-=\tfrac12(\mathrm{id}-S)
\end{equation}
are bounded complementary projections on $L^2(\Gamma)$, with
$\operatorname{im}P=H^2(\C_+)$ and $\operatorname{im}Q=H^2(\C_-)$ \cite{DKMVZ01}.

We seek the rows of $N$ separately.  Let $\varepsilon_i\in\C^{1\times2r}$ be the
$i$-th standard row vector and look for
\begin{equation}\label{eq:ansatz}
  N^{(i)}=\varepsilon_i+\mathcal C\phi^{(i)},
  \qquad \phi^{(i)}\in L^2\bigl(\Gamma,\C^{1\times2r}\bigr).
\end{equation}
Substituting $N^{(i)}_+=\varepsilon_i+P\phi^{(i)}$ and
$N^{(i)}_-=\varepsilon_i-Q\phi^{(i)}$ into $N^{(i)}_+=N^{(i)}_-v$ gives
$T\phi^{(i)}=\varepsilon_i(v-I_{2r})$, where
\begin{equation}\label{eq:defT}
  T\phi:=P\phi+(Q\phi)\,v,
  \qquad T:L^2\bigl(\Gamma,\C^{1\times2r}\bigr)\longrightarrow
              L^2\bigl(\Gamma,\C^{1\times2r}\bigr).
\end{equation}
The right-hand side lies in $L^2$ because $v-I_{2r}$ is bounded and supported on
the compact set $[a_0,b_g]$.  We say that $N$ is an \emph{$L^2$ solution} of RH
Problem \ref{def:rhp} if its rows are of the form \eqref{eq:ansatz} with
$T\phi^{(i)}=\varepsilon_i(v-I_{2r})$.

If $T\phi=0$, then $\widehat N=C\phi$ is a homogeneous $L^2$ solution of the
Riemann--Hilbert problem. Applying the vanishing lemma of Proposition \ref{prop:vanishing}
row-wise gives $\widehat N\equiv 0$, and hence
\[
\phi=\widehat N_+-\widehat N_-=0.
\]
Therefore
\begin{equation}\label{eq:ker0}
\ker T=\{0\}.
\end{equation}

In what follows we show $T$ in \eqref{eq:defT} is Fredholm of index zero. Combining  this with \eqref{eq:ker0} it then follows that $T$ is then invertible. Our strategy is the following. We reduce $T$ to a Toeplitz operator with piece-wise constant symbol. This then allows us to compute the index in terms of a winding number of a determinant. This computation miraculously turns to be  explicit.
\subsubsection{Reduction to a Toeplitz operator}

 Writing $\psi=\phi^T$ and
$u:=v^T$, the operator $T$ of \eqref{eq:defT} is similar to
\begin{equation}\label{eq:defA2}
  \mathcal{B}:=P+uQ \quad\text{on } L^2\bigl(\Gamma,\C^{2r}\bigr),
\end{equation}
so that $T$ and $\mathcal{B}$ are Fredholm together and have the same index.   

\begin{lemma}\label{lem:toeplitzreduction}
Let $b\in L^\infty(\Gamma)^{n\times n}$ with $b^{-1}\in L^\infty$, and let
$B:=P+bQ$ on $L^2(\Gamma,\C^n)$.  Put $c:=b^{-1}$ and let
$T(c):=Pc|_{\operatorname{im}P}:\operatorname{im}P\to\operatorname{im}P$ be the
Toeplitz operator with symbol $c$.  Then $B$ is Fredholm if and only if $T(c)$
is, and in that case $\operatorname{ind}B=\operatorname{ind}T(c)$.
\end{lemma}

\begin{proof}
Multiplication by $b^{-1}$ is boundedly invertible on $L^2$, so $B$ and
$b^{-1}B=cP+Q$ are Fredholm together with the same index.  Decompose
$L^2=\operatorname{im}P\oplus\operatorname{im}Q$ and let $f=f_++f_-$ be the
corresponding splitting.  Then
\[
  (cP+Q)f=cf_++f_-=\underbrace{Pcf_+}_{\in\,\operatorname{im}P}
  +\underbrace{Qcf_++f_-}_{\in\,\operatorname{im}Q},
\]
so that with respect to this decomposition
\[
  cP+Q=\begin{pmatrix}Pc|_{\operatorname{im}P}&0\\
                      Qc|_{\operatorname{im}P}&\mathrm{id}\end{pmatrix}
      =\begin{pmatrix}\mathrm{id}&0\\ Qc|_{\operatorname{im}P}&\mathrm{id}\end{pmatrix}
       \begin{pmatrix}Pc|_{\operatorname{im}P}&0\\0&\mathrm{id}\end{pmatrix}.
\]
The first factor is boundedly invertible, with inverse obtained by changing the
sign of the off-diagonal entry.  The claim follows.
\end{proof}

Applying Lemma \ref{lem:toeplitzreduction} to \eqref{eq:defA2} with $b=u=v^T$, we
must study the Toeplitz operator with symbol
\begin{equation}\label{eq:symbol}
  c:=u^{-1}=\bigl(v^{-1}\bigr)^{T}\in PC_{2r\times2r}(\Gamma).
\end{equation}

(Here $PC$ stands for piecewise continuous.) We now quote the one external result that we use for Fredholmness and for the
index.  Recall that for $c\in PC_{n\times n}(\Gamma)$ one sets
\begin{equation}\label{eq:hashsymbol}
  c^{\#}(\tau,\mu):=(1-\mu)\,c(\tau-0)+\mu\,c(\tau+0),
  \qquad \tau\in\Gamma,\ \mu\in[0,1],
\end{equation}
and lets $\det c^{\#}$ denote the closed curve in $\C$ obtained by traversing
$\tau\mapsto\det c(\tau)$ along $\Gamma$ and inserting, at each point of
discontinuity, the arc $\mu\mapsto\det c^{\#}(\tau,\mu)$.

\begin{theorem}\cite[\S5.48 (c)]{BS06}\label{thm:GK}
Let $c\in PC_{n\times n}(\Gamma)$ with $c^{\pm1}\in L^\infty(\Gamma)$.  Then
$T(c)$ is Fredholm on $\operatorname{im}P\cong H^2(\C_+)^n$ if and only if
\[
  \det c^{\#}(\tau,\mu)\ne0
  \qquad\text{for all }\tau\in\Gamma,\ \mu\in[0,1],
\]
and in that case
$\operatorname{ind}T(c)=-\operatorname{wind}\bigl(\det c^{\#}\bigr)$.
\end{theorem}

It turns out using Theorem \ref{thm:GK} we can explicitly compute the index of $T$ for the symbol $c$ \eqref{eq:symbol}.
\subsubsection{The index}
\begin{proposition}\label{prop:indexzero}
For every $K=(K_1,\dots,K_g)\in\Ur^g$ the operator $T$ of \eqref{eq:defT} is
Fredholm on $L^2(\Gamma,\C^{1\times2r})$ with
\[
  \operatorname{ind}T=0 .
\]
\end{proposition}

\begin{proof}
By the discussion above it suffices to prove that $T(c)$ is Fredholm of index
zero, with $c$ as in \eqref{eq:symbol}.  Since $\det X^T=\det X$, we have for
every $\tau$ and $\mu$
\begin{equation}\label{eq:detsymbol}
  \det c^{\#}(\tau,\mu)
  =\det\Bigl[(1-\mu)\,v(\tau-0)^{-1}+\mu\,v(\tau+0)^{-1}\Bigr],
\end{equation}
and we compute the right-hand side at each of the $2g+2$ discontinuities.  We
use repeatedly that
\begin{equation}\label{eq:silvester}
  \det\begin{pmatrix}\mathsf A&\mathsf B\\\mathsf C&\mathsf D\end{pmatrix}
  =\det\bigl(\mathsf{AD}-\mathsf{BC}\bigr)
  \qquad\text{whenever }\mathsf C\mathsf D=\mathsf D\mathsf C,
\end{equation}
and that $J^{-1}=-J$ by \eqref{eq:CLambda}.

Let $\tau=b_{j-1}$, so that $v(\tau-0)=J$ and $v(\tau+0)=D_j$.  Then
\[
  (1-\mu)J^{-1}+\mu D_j^{-1}
  =\begin{pmatrix}\mu K_j^{-1}&-(1-\mu)I_r\\[2pt]
                  (1-\mu)I_r&\mu K_j\end{pmatrix},
\]
whose lower two blocks commute, so that by \eqref{eq:silvester}
\[
  \det c^{\#}(b_{j-1},\mu)
  =\det\Bigl(\mu^2K_j^{-1}K_j+(1-\mu)^2I_r\Bigr)
  =\bigl(\mu^2+(1-\mu)^2\bigr)^{r}.
\]
Let $\tau=a_j$, so that $v(\tau-0)=D_j$ and $v(\tau+0)=J$.  Then
\[
  (1-\mu)D_j^{-1}+\mu J^{-1}
  =\begin{pmatrix}(1-\mu)K_j^{-1}&-\mu I_r\\[2pt]
                  \mu I_r&(1-\mu)K_j\end{pmatrix},
\]
and again the lower blocks commute, so
$\det c^{\#}(a_j,\mu)=\bigl(\mu^2+(1-\mu)^2\bigr)^{r}$.  At $\tau=a_0$ we have
$v(\tau-0)=I_{2r}$, $v(\tau+0)=J$, and
\[
  (1-\mu)I_{2r}+\mu J^{-1}
  =\begin{pmatrix}(1-\mu)I_r&-\mu I_r\\ \mu I_r&(1-\mu)I_r\end{pmatrix},
\]
with determinant $\bigl(\mu^2+(1-\mu)^2\bigr)^{r}$; the computation at
$\tau=b_g$ is the same with $\mu$ replaced by $1-\mu$.  At every point of
continuity $\det c^{\#}(\tau,\mu)=\det v(\tau)^{-1}$, and
\[
  \det J=(-1)^{r^2}\det(I_r)\det(-I_r)=(-1)^{r^2+r}=1,
  \qquad
  \det D_j=\det K_j\,\det K_j^{-1}=1,
\]
so $\det c\equiv1$ on $\Gamma$.  Consequently
\begin{equation}\label{eq:detrange}
  \det c^{\#}(\tau,\mu)=\bigl(\mu^2+(1-\mu)^2\bigr)^{r}
  \in\bigl[2^{-r},1\bigr]
  \qquad\text{for all }\tau\in\Gamma,\ \mu\in[0,1].
\end{equation}
In particular $\det c^{\#}$ never vanishes, so $T(c)$ is Fredholm by Theorem
\ref{thm:GK}.  Moreover the closed curve $\det c^{\#}$ lies entirely in the
interval $[2^{-r},1]$ of the positive real axis, and therefore has winding number $0$ about the
origin.  By Theorem \ref{thm:GK}, $\operatorname{ind}T(c)=0$.
\end{proof}

\section{Reduction to RH problem \ref{def:rhp} and strong asymptotics}\label{subsubsec:reduction}

\subsection{Jacobi type weights}
\label{subsec:jacobi}
 
We carry out the steepest descent steps for the weight of Assumption
\ref{ass:model1} and show that the global problem that remains is RH Problem
\ref{def:rhp}.  The transformations themselves are the several cut version of
\cite[\S3]{DKR23}, and of \cite[\S\S3--5]{KMVV04} in the scalar case; we keep
them brief and concentrate on the point at which the multivaluedness of the
Szeg\H{o} function enters.
 
\subsubsection{The Szeg\H{o} function of the weight}
 
We first check that Theorem \ref{thm:thm:WM} is applicable.  It is convenient
to split off the scalar factor, as is done for $[-1,1]$ in \cite[\S2]{DKR23}.
 
\begin{lemma}\label{lem:jacobiszego}
Let $W$ be as in Assumption \ref{ass:model1}.  Then there is a
$D:\Omega_0\to\C^{r\times r}$, holomorphic and pointwise invertible with
$D(\infty)>0$, satisfying \eqref{eq:band-main} and \eqref{eq:gap-main} with
$M=W$.
\end{lemma}
 
\begin{proof}
Consider first $H$.  Its entries are bounded on $E$ and so lie in
$L^1(\partial\Omega,\omega_\infty)$.  By Assumption \ref{ass:model1}(c),(d) the
function $\det H$ is holomorphic near $E$, positive on $E^\circ$ and not
identically zero, so at each $e\in\mathcal E$ we may write
$\det H(x)=|x-e|^{\kappa_e}\rho_e(x)$ with $\kappa_e\ge0$ and $\rho_e$
bounded away from $0$; hence $|\log\det H(x)|\le C\bigl(1+|\log|x-e||\bigr)$
near $e$.  Since $\omega_\infty$ has a density comparable to $|x-e|^{-1/2}$ at
each endpoint, $\log\det H\in L^1(\partial\Omega,\omega_\infty)$.  Thus
Assumption \ref{ass:1} holds for $H$, and Theorem \ref{thm:thm:WM} produces
$D_H$ with $H=D_{H\pm}D_{H\pm}^*$ on $E$ and $D_{H+}=D_{H-}U_j^H$ on $\Sigma_j$,
with $U_j^H\in\Ur$.
 
For the scalar factor $h$ the same computation gives
$\log h\in L^1(\partial\Omega,\omega_\infty)$, and the proof of Theorem
\ref{thm:thm:WM} applies verbatim with $r=1$, with Theorem \ref{thm:WM}
replaced by the classical scalar factorisation on the disc: the outer function
\[
  G_h(w)=\exp\Bigl(\tfrac12\int_{\T}\frac{\tau+w}{\tau-w}
    \log\bigl(h\circ p^*\bigr)(\tau)\,dm(\tau)\Bigr),
  \qquad w\in\D,
\]
is well defined and zero free as soon as $\log(h\circ p^*)\in L^1(m)$, it
satisfies $|G_h^*|^2=h\circ p^*$ almost everywhere on $\T$, and it is unique up
to a unimodular constant.  Invariance of $h\circ p^*$ under the deck group
therefore gives $G_h\circ\gamma=u_\gamma G_h$ with $|u_\gamma|=1$, and
$D_h:=G_h\circ s$ satisfies $h=|D_{h\pm}|^2$ on $E$ and $D_{h+}=u_jD_{h-}$ on
$\Sigma_j$ with $|u_j|=1$.
 
Since $D_h$ is scalar it commutes, so $D_hD_H$ satisfies
$W=hH=D_\pm D_\pm^*$ and $D_+=D_-U_j$ with $U_j=u_jU_j^H\in\Ur$.  Multiplying on
the right by a constant unitary matrix as in Lemma \ref{lem:canonical-normalization} we may
arrange $D(\infty)>0$.
\end{proof}

\subsubsection{Transformation with \texorpdfstring{$g$}{g} function}
 
Let $\mu_E$ be the equilibrium measure of the compact set $E$.  It coincides
with the harmonic measure $\omega_\infty$ of $\Omega$ at $\infty$ that appears
in Assumption \ref{ass:1}.  Put
\begin{equation}\label{eq:jacobig}
  g(z)=\int\log(z-s)\,d\mu_E(s),
  \qquad z\in\C\setminus(-\infty,b_g],
\end{equation}
with principal branches, and $\ell=2\log\operatorname{cap}(E)$.  We shall use the
following standard properties, see \cite[Ch.~I]{ST97}.
 
\begin{lemma}\label{lem:jacobig}
\begin{enumerate}
\item[(i)] $g(z)=\log z+O(z^{-1})$ as $z\to\infty$.
\item[(ii)] $g_++g_-=\ell$ on $E$.
\item[(iii)] On $\Sigma_j$ the difference $g_+-g_-$ is the constant $i\omega_j$,
  where
  \begin{equation}\label{eq:jacobiomega}
    \omega_j:=2\pi\,\omega_\infty\bigl([a_j,b_g]\bigr)
    =2\pi\sum_{\ell=j}^{g}\omega_\infty\bigl([a_\ell,b_\ell]\bigr)\in(0,2\pi),
    \qquad j=1,\dots,g.
  \end{equation}
  Moreover $g_+-g_-=2\pi i$ on $(-\infty,a_0)$ and $g_+=g_-$ on $(b_g,\infty)$.
\item[(iv)] $\ell-2\operatorname{Re}g(z)=-2\,g_\Omega(z,\infty)<0$ for
  $z\in\Omega$, where $g_\Omega(\cdot,\infty)$ is the Green function of $\Omega$
  with pole at infinity.
\end{enumerate}
\end{lemma}
 
Statement (iii) is immediate from $\log_\pm(x-s)=\log|x-s|\pm i\pi$ for $s>x$,
which gives $g_+-g_-=2\pi i\,\mu_E\bigl((x,\infty)\bigr)$, a constant on each
gap since $\mu_E$ puts no mass there.  Statement (iv) is Frostman's theorem
together with $g_\Omega(z,\infty)=\log\frac1{\operatorname{cap}(E)}-U^{\mu_E}(z)$.
 
The first transformation is
\begin{equation}\label{eq:jacobiT}
  T(z)=e^{-\frac{n\ell}{2}\sigma_3}\,Y(z)\,
       e^{-n\left(g(z)-\frac{\ell}{2}\right)\sigma_3},
\end{equation}
with $Y$ the solution of RH Problem \ref{rhp:model1}.
 
\begin{lemma}\label{lem:jacobiTjump}
$T$ is analytic in $\C\setminus[a_0,b_g]$, $T(z)=I_{2r}+O(z^{-1})$ as
$z\to\infty$, it has the same behavior as $Y$ at the endpoints, and
\begin{equation}\label{eq:jacobiTjumps}
  T_+=T_-\times
  \begin{cases}
    \begin{pmatrix}e^{-n(g_+-g_-)}I_r&W\\0_r&e^{n(g_+-g_-)}I_r\end{pmatrix},
      &x\in(a_\ell,b_\ell),\quad\ell=0,\dots,g,\\[14pt]
    e^{-in\omega_j\sigma_3},&x\in\Sigma_j,\quad j=1,\dots,g.
  \end{cases}
\end{equation}
\end{lemma}
 
\begin{proof}
Substituting $Y_-=e^{\frac{n\ell}{2}\sigma_3}T_-
e^{n(g_--\frac{\ell}{2})\sigma_3}$ into \eqref{eq:model1jump} gives
\[
  T_+=T_-\begin{pmatrix}e^{-n(g_+-g_-)}I_r&e^{n(g_++g_--\ell)}W\\
                        0_r&e^{n(g_+-g_-)}I_r\end{pmatrix}
\]
on a cut, and the exponent in the upper right corner vanishes by Lemma
\ref{lem:jacobig}(ii).  On $\Sigma_j$ the matrix $Y$ has no jump, so
$T_+=T_-e^{-n(g_+-g_-)\sigma_3}=T_-e^{-in\omega_j\sigma_3}$ by Lemma
\ref{lem:jacobig}(iii).  The same computation on $(-\infty,a_0)$ gives the
factor $e^{-2\pi in\sigma_3}=I_{2r}$, and on $(b_g,\infty)$ it gives $I_{2r}$;
hence $T$ is analytic across $\R\setminus[a_0,b_g]$.  Finally
$e^{-n(g(z)-\ell/2)\sigma_3}z^{n\sigma_3}=e^{n\ell\sigma_3/2}(I_{2r}+O(z^{-1}))$
as $z\to\infty$ by Lemma \ref{lem:jacobig}(i), which together with
\eqref{eq:model1infty} gives the normalisation.
\end{proof}
 
\subsubsection{Opening of lenses}
 
Put
\begin{equation}\label{eq:jacobiphi}
  \phi(z)=\ell-2g(z),
\end{equation}
which is analytic in $\C\setminus(-\infty,b_g]$ and in particular in
$\C\setminus\R$.  By Lemma \ref{lem:jacobig}(ii) we have $\phi_++\phi_-=0$ on
each cut, so that $\phi_\pm$ is purely imaginary there and
$\phi_+=-(g_+-g_-)$; and by Lemma \ref{lem:jacobig}(iv),
\begin{equation}\label{eq:jacobiRephi}
  \operatorname{Re}\phi(z)=-2\,g_\Omega(z,\infty)<0,
  \qquad z\in\Omega .
\end{equation}
Using $e^{n\phi_+}e^{n\phi_-}=1$ we factor the jump matrix on a cut as
\begin{equation}\label{eq:jacobifactorisation}
  \begin{pmatrix}e^{n\phi_+}I_r&W\\0_r&e^{n\phi_-}I_r\end{pmatrix}
  =\begin{pmatrix}I_r&0_r\\W^{-1}e^{n\phi_-}&I_r\end{pmatrix}
   \begin{pmatrix}0_r&W\\-W^{-1}&0_r\end{pmatrix}
   \begin{pmatrix}I_r&0_r\\W^{-1}e^{n\phi_+}&I_r\end{pmatrix},
\end{equation}
open a lens around each cut, and set
\begin{equation}\label{eq:jacobiS}
  S(z)=T(z)\times
  \begin{cases}
    \begin{pmatrix}I_r&0_r\\-W(z)^{-1}e^{n\phi(z)}&I_r\end{pmatrix},
      &\text{in the upper parts of the lenses},\\[10pt]
    \begin{pmatrix}I_r&0_r\\ \ \ W(z)^{-1}e^{n\phi(z)}&I_r\end{pmatrix},
      &\text{in the lower parts of the lenses},\\[10pt]
    I_{2r},&\text{elsewhere}.
  \end{cases}
\end{equation}
Here $W$ denotes the analytic continuation of the weight off the cut, obtained
by continuing each factor of \eqref{eq:model1weight}: on a neighborhood
$\mathcal U_\ell$ of $[a_\ell,b_\ell]$ meeting no other cut, the product
$\prod_{k\ne\ell}|z-a_k|^{\alpha_k}|z-b_k|^{\beta_k}$ is positive and real
analytic on $\mathcal U_\ell\cap\R$ and so continues analytically, while
$(z-a_\ell)^{\alpha_\ell}(b_\ell-z)^{\beta_\ell}$ is analytic on
$\mathcal U_\ell\setminus\bigl(\R\setminus(a_\ell,b_\ell)\bigr)$ with principal
branches.  This is the several cut version of \cite[\S3.3]{DKR23}.
 
By \eqref{eq:jacobifactorisation} the matrix $S$ has the constant jump
$\begin{psmallmatrix}0_r&W\\-W^{-1}&0_r\end{psmallmatrix}$ on $E^\circ$, the
jump $e^{-in\omega_j\sigma_3}$ on $\Sigma_j$, and by \eqref{eq:jacobiRephi} the
jumps on the boundary of the lenses are $I_{2r}+O(e^{-cn})$ for some $c>0$,
uniformly on compact subsets of the boundary away from $\mathcal E$.  We note that
\eqref{eq:jacobiRephi} holds on all of $\Omega$, so that the width of the lenses
is restricted only by the domain of analyticity of $H$.
 
\subsubsection{The global problem and its reduction}
 
Discarding the jumps that tend to the identity leads to the following problem.
 
\begin{rhp}\label{rhp:jacobiouter}
Let $n\ge1$.  We seek $\mathcal N=\mathcal N(\cdot\,;n)$ such that
\begin{enumerate}
\item $\mathcal N:\C\setminus[a_0,b_g]\to\C^{2r\times2r}$ is analytic.
\item On each cut,
  \begin{equation}\label{eq:jacobiouterband}
    \mathcal N_+=\mathcal N_-
    \begin{pmatrix}0_r&W\\-W^{-1}&0_r\end{pmatrix},
    \qquad x\in(a_\ell,b_\ell),\quad\ell=0,\dots,g.
  \end{equation}
\item On each gap,
  \begin{equation}\label{eq:jacobioutergap}
    \mathcal N_+=\mathcal N_-\,e^{-in\omega_j\sigma_3},
    \qquad x\in\Sigma_j,\quad j=1,\dots,g.
  \end{equation}
\item $\mathcal N(z)=I_{2r}+O(z^{-1})$ as $z\to\infty$.
\item $\mathcal N(z)\mathsf D(z)^{-1}=O\bigl(|z-e|^{-1/4}\bigr)$ as $z\to e$,
  $e\in\mathcal E$, with $\mathsf D$ as in \eqref{eq:jacobiconjugators} below.
\end{enumerate}
\end{rhp}
 
Let $D$ be the Szeg\H{o} function of Lemma \ref{lem:jacobiszego} and set
\begin{equation}\label{eq:jacobisharp}
  D^\sharp(z):=D(\bar z)^*,\qquad z\in\Omega_0 .
\end{equation}
As it involves two anti-holomorphic conjugations and $\Omega_0$ is invariant
under complex conjugation, $D^\sharp$ is again holomorphic and pointwise
invertible on $\Omega_0$, with boundary values
\begin{equation}\label{eq:jacobisharpbv}
  D^\sharp_+(x)=D_-(x)^*,\qquad D^\sharp_-(x)=D_+(x)^*,
  \qquad x\in(a_0,b_g).
\end{equation}
Writing $X^{-*}=(X^{-1})^*$ we put
\begin{equation}\label{eq:jacobiconjugators}
  \mathsf D(z)=\begin{pmatrix}D(z)^{-1}&0_r\\0_r&D^\sharp(z)\end{pmatrix},
  \qquad
  \mathsf D_\infty=\begin{pmatrix}D(\infty)&0_r\\0_r&D(\infty)^{-*}\end{pmatrix}.
\end{equation}
 
\begin{proposition}\label{prop:jacobireduction}
Let $U_1,\dots,U_g\in\Ur$ be as in \eqref{eq:gap-main} for the weight $W$, let
$\omega_1,\dots,\omega_g$ be as in \eqref{eq:jacobiomega}, and put
\begin{equation}\label{eq:jacobiK}
  K_j=e^{-in\omega_j}U_j\in\Ur,\qquad j=1,\dots,g.
\end{equation}
If $N$ solves RH Problem \ref{def:rhp} with these data, then
\begin{equation}\label{eq:jacobiglobal}
  \mathcal N(z)=\mathsf D_\infty\,N(z)\,\mathsf D(z)
\end{equation}
solves RH Problem \ref{rhp:jacobiouter}; conversely every solution of RH Problem
\ref{rhp:jacobiouter} is of this form.  In particular RH Problem
\ref{rhp:jacobiouter} has exactly one solution, by Propositions
\ref{prop:Nsolves} and \ref{prop:unique}.
\end{proposition}
 
\begin{proof}
Each of $N$, $D^{-1}$ and $D^\sharp$ is analytic in $\C\setminus[a_0,b_g]$, and
so is $\mathcal N$; item 5 is (N4).  For item 4, $D$ is holomorphic at
$\infty\in\Omega_0$, so $D(z)=D(\infty)+O(z^{-1})$ and
$D^\sharp(z)=D(\infty)^*+O(z^{-1})$, whence
\[
  \mathsf D_\infty\mathsf D(z)
  =\begin{pmatrix}D(\infty)D(\infty)^{-1}&0_r\\
                  0_r&D(\infty)^{-*}D(\infty)^*\end{pmatrix}+O(z^{-1})
  =I_{2r}+O(z^{-1}),
\]
and $N(z)=I_{2r}+O(z^{-1})$ by (N3).
 
For the jumps we compute, using \eqref{eq:jacobiglobal} and
\eqref{eq:jacobisharpbv},
\[
  \mathcal N_-^{-1}\mathcal N_+
  =\mathsf D_-^{-1}\bigl(N_-^{-1}N_+\bigr)\mathsf D_+,
  \qquad
  \mathsf D_-^{-1}=\begin{pmatrix}D_-&0_r\\0_r&D_+^{-*}\end{pmatrix},
  \qquad
  \mathsf D_+=\begin{pmatrix}D_+^{-1}&0_r\\0_r&D_-^*\end{pmatrix}.
\]
On a cut, $N_-^{-1}N_+=\begin{psmallmatrix}0_r&I_r\\-I_r&0_r\end{psmallmatrix}$
by (N1), and therefore
\[
  \mathcal N_-^{-1}\mathcal N_+
  =\begin{pmatrix}D_-&0_r\\0_r&D_+^{-*}\end{pmatrix}
   \begin{pmatrix}0_r&I_r\\-I_r&0_r\end{pmatrix}
   \begin{pmatrix}D_+^{-1}&0_r\\0_r&D_-^*\end{pmatrix}
  =\begin{pmatrix}0_r&D_-D_-^*\\-D_+^{-*}D_+^{-1}&0_r\end{pmatrix},
\]
which is $\begin{psmallmatrix}0_r&W\\-W^{-1}&0_r\end{psmallmatrix}$ by
\eqref{eq:band-main}, since $D_-D_-^*=W$ and
$D_+^{-*}D_+^{-1}=(D_+D_+^*)^{-1}=W^{-1}$.  This is
\eqref{eq:jacobiouterband}.
 
On $\Sigma_j$, $N_-^{-1}N_+=\operatorname{diag}(K_j,K_j^{-1})$ by (N2), so
\[
  \mathcal N_-^{-1}\mathcal N_+
  =\begin{pmatrix}D_-K_jD_+^{-1}&0_r\\0_r&D_+^{-*}K_j^{-1}D_-^*\end{pmatrix}.
\]
By \eqref{eq:gap-main} we have $D_+=D_-U_j$, hence $D_+^{-1}=U_j^{-1}D_-^{-1}$
and $D_+^{-*}=D_-^{-*}U_j$, so that with \eqref{eq:jacobiK} the two diagonal
blocks are
\[
  D_-K_jU_j^{-1}D_-^{-1}=e^{-in\omega_j}I_r,
  \qquad
  D_-^{-*}U_jK_j^{-1}D_-^*=e^{in\omega_j}I_r .
\]
This is \eqref{eq:jacobioutergap}.  All the steps are reversible, which gives
the converse.
\end{proof}

Proposition \ref{prop:jacobireduction} is the global part of the analysis.  What
remains is local, and is the one cut analysis repeated $2g+2$ times.
\paragraph{Local parametrices.}
Choose pairwise disjoint discs $B_e$ around the endpoints
$e\in\mathcal E$, sufficiently small. Let $\Sigma_S$ denote the jump
contour of $S$. In $B_e$ we seek a local parametrix $P_e$ satisfying
the following RH problem.
\begin{rhp}

\begin{enumerate}
\item
$P_e$ is analytic in $B_e\setminus\Sigma_S$.

\item
$P_e$ has the same jumps as $S$ in $B_e$. Thus,
\[
(P_e)_+=(P_e)_-
\begin{pmatrix}
0_r&W\\
-W^{-1}&0_r
\end{pmatrix}
\]
on the part of the adjacent cut contained in $B_e$, and
\[
(P_e)_+=(P_e)_-
\begin{pmatrix}
I_r&0_r\\
e^{n\phi}W^{-1}&I_r
\end{pmatrix}
\]
on the boundary of the lens in $B_e$. If $e$ is an internal endpoint
adjacent to the gap $\Sigma_j$, then in addition
\[
(P_e)_+=(P_e)_-e^{-in\omega_j\sigma_3}
\qquad\text{on }B_e\cap\Sigma_j .
\]

\item
As $n\to\infty$,
\[
P_e(z)\mathcal N(z)^{-1}
   =I_{2r}+O(n^{-1})
\]
uniformly for $z\in\partial B_e\setminus\Sigma_S$.

\item
The matrix $S(z)P_e(z)^{-1}$ remains bounded as $z\to e$.
\end{enumerate}
\end{rhp}
We reduce this problem to the one considered in
\cite[\S3.5]{DKR23}. Let $V_e$ be the local factor of the weight
defined as in \cite[\S3.5.2]{DKR23}, and seek $P_e$ in the form
\begin{equation}\label{eq:localPe}
P_e(z)=E_{e,n}(z)P_e^{(1)}(z)
\begin{pmatrix}
e^{n\phi(z)/2}V_e(z)^{-1}&0_r\\
0_r&e^{-n\phi(z)/2}V_e^\sharp(z)
\end{pmatrix},
\end{equation}
where $E_{e,n}$ is analytic and invertible in $B_e$.

The calculation in \cite[Lemma 3.5]{DKR23} applies without change on
the adjacent cut and on the boundary of the lens. At an 
endpoint, the last factor in \eqref{eq:localPe} has the jump
$e^{-in\omega_j\sigma_3}$ on $\Sigma_j$, since
$\phi_+-\phi_-=-2i\omega_j$. Hence $P_e^{(1)}$ has no jump on the
adjacent gap. It follows that $P_e^{(1)}$ satisfies exactly the
constant jump conditions in \cite[(3.30)]{DKR23}, with Bessel orders
\[
\nu_{e,k}
 =\gamma_e+\operatorname{ord}_e\lambda_{e,k},
\qquad k=1,\ldots,r,
\]
where $\lambda_{e,1},\ldots,\lambda_{e,r}$ are the local analytic
eigenvalues of $H$.

With these local parametrices, the final transformation is a
small-norm RH problem, exactly as in \cite[\S3.8]{DKR23}. Hence by the same strategy,
\[
R(z)=I_{2r}+O\!\left(\frac{1}{n(1+|z|)}\right)
\]
uniformly away from its jump contour. Reversing
\eqref{eq:jacobiT} and \eqref{eq:jacobiS} gives, uniformly on compact
subsets of $\C\setminus [a_0,b_g]$,
\begin{equation}\label{eq:MVOP1}
e^{-ng(z)}P_n(z)
 =D(\infty)\cc(z)
  F_1\bigl(z;K_1,\ldots,K_g\bigr)D(z)^{-1}
  +O(n^{-1}).
\end{equation}

 \subsection{Varying exponential weights}
\label{subsec:varying}
 
We now do the same for the weight of Assumption \ref{ass:model2}.  The
transformations are those of \cite[\S5]{DR25} in the one cut case and of
\cite{DKMVZ99b} in the scalar several cut case.  Two things change with respect
to Section \ref{subsec:jacobi}: the set $E$ is not the set of orthogonality but
the support of $\mu_V$, so that jumps survive on $\R\setminus E$, and the
inequality that makes the lenses work is only local.  The conjugation by the
Szeg\H{o} function, on the other hand, is identical.
 
\subsubsection{The Szeg\H{o} function of the matrix part}

\begin{lemma}\label{lem:varyingszego}
Let $M$ be as in Assumption \ref{ass:model2}.  Then Assumption \ref{ass:1} is
satisfied by $M|_E$, so that Theorem \ref{thm:thm:WM} provides a matrix
Szeg\H{o} function $D$ with $D(\infty)>0$, unitary matrices $U_1,\dots,U_g$ as
in \eqref{eq:gap-main}, and moreover $D$ and $D^{-1}$ are bounded on $\Omega_0$.
\end{lemma}
 
\begin{proof}
$M$ is continuous and Hermitian positive definite on the compact set $E$, so
there are $0<c\le C$ with $cI_r\le M(x)\le CI_r$ for $x\in E$.  In particular
the entries of $M$ are bounded, hence in $L^1(\partial\Omega,\omega_\infty)$,
and $\log\det M$ is bounded, hence in $L^1(\partial\Omega,\omega_\infty)$.  This
is Assumption \ref{ass:1}.
 
For the last statement, write $\mathcal M=M\circ p^*$ as in \eqref{eq:lifted-weight}
and let $G_0$ be the Wiener--Masani factor of Theorem \ref{thm:thm:WM}.  From
$\mathcal M=G_0G_0^*$ almost everywhere on $\T$ we get $\|G_0^*\|^2\le C$ there,
so the entries of $G_0$ lie in $H^2(\D)$ and have essentially bounded boundary
values, hence lie in $H^\infty(\D)$.  Also $\det G_0$ is outer with
$|\det G_0|^2=\det\mathcal M\ge c^{\,r}$ almost everywhere, so $\log|\det G_0|$
is the Poisson integral of a function bounded below by $\frac r2\log c$ and
therefore $|\det G_0|\ge c^{\,r/2}$ on $\D$.  Consequently
$G_0^{-1}=(\det G_0)^{-1}\operatorname{adj}G_0$ is bounded on $\D$ as well.
Since $D=G\circ s$ with $s$ as in \eqref{eq:section-s}, the bounds pass to $D$
and $D^{-1}$.
\end{proof}
 
\subsubsection{Transformation with \texorpdfstring{$g$}{g} function}
 
Let $\mu_V$ be the equilibrium measure of Assumption \ref{ass:model2}(b), put
\begin{equation}\label{eq:varyingg}
  g(z)=\int\log(z-s)\,d\mu_V(s),\qquad z\in\C\setminus(-\infty,b_g],
\end{equation}
with principal branches, and let $\ell$ be the Euler--Lagrange constant.  We use
the following, for which we refer to \cite[Ch.~I]{ST97} and
\cite[\S3]{DKMVZ99b}.
 
\begin{lemma}\label{lem:varyingg}
\begin{enumerate}
\item[(i)] $g(z)=\log z+O(z^{-1})$ as $z\to\infty$.
\item[(ii)] $g_++g_--V-\ell=0$ on $E$, and $g_++g_--V-\ell<0$ on $\R\setminus E$.
\item[(iii)] On $\Sigma_j$ the difference $g_+-g_-$ is the constant $i\omega_j$,
  where
  \begin{equation}\label{eq:varyingomega}
    \omega_j:=2\pi\,\mu_V\bigl([a_j,b_g]\bigr)
    =2\pi\sum_{\ell=j}^{g}\mu_V\bigl([a_\ell,b_\ell]\bigr)\in(0,2\pi),
    \qquad j=1,\dots,g;
  \end{equation}
  moreover $g_+-g_-=2\pi i$ on $(-\infty,a_0)$ and $g_+=g_-$ on $(b_g,\infty)$.
\end{enumerate}
\end{lemma}
 
As in Section \ref{subsec:jacobi}, (iii) follows from
$g_+-g_-=2\pi i\,\mu_V\bigl((x,\infty)\bigr)$ together with the fact that
$\mu_V$ puts no mass on a gap.  Note that $g_++g_-=2\int\log|x-s|\,d\mu_V(s)$ is
real on $\R$, a fact we use below.  We set
\begin{equation}\label{eq:varyingphi}
  \phi(z)=V(z)+\ell-2g(z),
\end{equation}
which by Assumption \ref{ass:model2}(a) is analytic in a neighborhood of $\R$
with $(-\infty,b_g]$ removed, and in particular on the intersection of that
neighborhood with $\C\setminus\R$.  By Lemma \ref{lem:varyingg}(ii),
\begin{equation}\label{eq:varyingphiprops}
  \phi_++\phi_-=0 \ \text{ on } E,
  \qquad
  \phi_++\phi_->0 \ \text{ on } \R\setminus E,
\end{equation}
so that on a cut $\phi_\pm$ is purely imaginary with $\phi_+=-(g_+-g_-)$, while
on $\R\setminus E$ the real part $\phi_++\phi_-$ is strictly positive.
 
The first transformation is again
\begin{equation}\label{eq:varyingT}
  T(z)=e^{-\frac{n\ell}{2}\sigma_3}\,Y(z)\,
       e^{-n\left(g(z)-\frac{\ell}{2}\right)\sigma_3},
\end{equation}
now with $Y$ the solution of RH Problem \ref{rhp:model2}.  
 
\begin{lemma}\label{lem:varyingTjump}
$T$ is analytic in $\C\setminus\R$, $T(z)=I_{2r}+O(z^{-1})$ as $z\to\infty$, and
\begin{equation}\label{eq:varyingTjumps}
  T_+=T_-\times
  \begin{cases}
    \begin{pmatrix}e^{n\phi_+}I_r&M\\0_r&e^{-n\phi_+}I_r\end{pmatrix},
      &x\in(a_\ell,b_\ell),\quad\ell=0,\dots,g,\\[14pt]
    e^{-in\omega_j\sigma_3}
    \begin{pmatrix}I_r&e^{in\omega_j}e^{-\frac n2(\phi_++\phi_-)}M\\
                   0_r&I_r\end{pmatrix},
      &x\in\Sigma_j,\quad j=1,\dots,g,\\[14pt]
    \begin{pmatrix}I_r&e^{-\frac n2(\phi_++\phi_-)}M\\0_r&I_r\end{pmatrix},
      &x\in\R\setminus[a_0,b_g].
  \end{cases}
\end{equation}
\end{lemma}
 
\begin{proof}
Substituting $Y_-=e^{\frac{n\ell}{2}\sigma_3}T_-
e^{n(g_--\frac{\ell}{2})\sigma_3}$ into \eqref{eq:model2jump} gives, for
$x\in\R$,
\[
  T_+=T_-\begin{pmatrix}
    e^{-n(g_+-g_-)}I_r&e^{n(g_++g_--V-\ell)}M\\
    0_r&e^{n(g_+-g_-)}I_r\end{pmatrix},
\]
where we used $W_n=e^{-nV}M$.  Now $g_++g_--V-\ell=-\frac12(\phi_++\phi_-)$ and
$-(g_+-g_-)=\frac12(\phi_+-\phi_-)$ by \eqref{eq:varyingphi}.  On a cut the
first of these vanishes and $\phi_+=-\phi_-$, which gives the first line.  On
$\Sigma_j$ the diagonal entries are $e^{\mp in\omega_j}$ by Lemma
\ref{lem:varyingg}(iii), and factoring them out on the left gives the second
line.  On $(-\infty,a_0)$ the diagonal entries are $e^{\mp2\pi in}=1$ and on
$(b_g,\infty)$ they equal $1$ as well, which gives the third line.  The
behavior at infinity follows from \eqref{eq:model2infty} and Lemma
\ref{lem:varyingg}(i).
\end{proof}
 
By \eqref{eq:varyingphiprops} the off-diagonal entries in the second and third
lines of \eqref{eq:varyingTjumps} are exponentially small in $n$, uniformly on
compact subsets of $\R\setminus E$; that they remain so as $x\to\pm\infty$
follows from Assumption \ref{ass:model2}(a) and (d), since $\phi(x)$ grows like
$V(x)-2\log|x|\to\infty$ while $\|M(x)\|$ grows only polynomially.  This is the
point at which the growth hypothesis on $V$ is used; see \cite[\S4]{DKMVZ99b}.
 
\subsubsection{Opening of lenses}
 
On a cut we factor, exactly as in \eqref{eq:jacobifactorisation} and again using
$e^{n\phi_+}e^{n\phi_-}=1$,
\begin{equation}\label{eq:varyingfactorisation}
  \begin{pmatrix}e^{n\phi_+}I_r&M\\0_r&e^{n\phi_-}I_r\end{pmatrix}
  =\begin{pmatrix}I_r&0_r\\M^{-1}e^{n\phi_-}&I_r\end{pmatrix}
   \begin{pmatrix}0_r&M\\-M^{-1}&0_r\end{pmatrix}
   \begin{pmatrix}I_r&0_r\\M^{-1}e^{n\phi_+}&I_r\end{pmatrix},
\end{equation}
and we set
\begin{equation}\label{eq:varyingS}
  S(z)=T(z)\times
  \begin{cases}
    \begin{pmatrix}I_r&0_r\\-M(z)^{-1}e^{n\phi(z)}&I_r\end{pmatrix},
      &\text{in the upper parts of the lenses},\\[10pt]
    \begin{pmatrix}I_r&0_r\\ \ \ M(z)^{-1}e^{n\phi(z)}&I_r\end{pmatrix},
      &\text{in the lower parts of the lenses},\\[10pt]
    I_{2r},&\text{elsewhere}.
  \end{cases}
\end{equation}
By Assumption \ref{ass:model2}(c) and (d), $M$ is analytic and invertible on a
neighborhood $\mathcal U$ of $E$, so no branch cuts are involved and the lenses
may be taken inside $\mathcal U$.
 
The relevant inequality is now local.  From $\phi_+(x)=-2\pi
i\,\mu_V\bigl((x,\infty)\bigr)$ we get $\phi_+'(x)=2\pi i\psi(x)$, and therefore
\[
  \operatorname{Re}\phi(x\pm i\varepsilon)=-2\pi\varepsilon\,\psi(x)
    +O(\varepsilon^2),
  \qquad \varepsilon\downarrow0,
\]
so that $\operatorname{Re}\phi<0$ on both sides of $E^\circ$, since $\psi>0$
there.  By the regularity assumed in \ref{ass:model2}(b) the lenses can be
chosen so that $\operatorname{Re}\phi\le-c<0$ on their boundary, away from the discs
around $\mathcal E$; see \cite[\S4]{DKMVZ99b}.  In contrast with Section
\ref{subsec:jacobi}, where $\operatorname{Re}\phi<0$ held on all of $\Omega$,
here the lenses have to be kept close to $E$.
 
Thus $S$ has the constant jump
$\begin{psmallmatrix}0_r&M\\-M^{-1}&0_r\end{psmallmatrix}$ on $E^\circ$, the
jump $e^{-in\omega_j\sigma_3}\bigl(I_{2r}+O(e^{-cn})\bigr)$ on $\Sigma_j$, and
jumps $I_{2r}+O(e^{-cn})$ on the boundary of the lenses and on $\R\setminus[a_0,b_g]$.
 
\subsubsection{The global problem and its reduction}
 
Discarding the jumps that tend to the identity we arrive at the following.
 
\begin{rhp}\label{rhp:varyingouter}
Let $n\ge1$.  We seek $\mathcal N=\mathcal N(\cdot\,;n)$ such that
\begin{enumerate}
\item $\mathcal N:\C\setminus[a_0,b_g]\to\C^{2r\times2r}$ is analytic.
\item $\displaystyle
  \mathcal N_+=\mathcal N_-\begin{pmatrix}0_r&M\\-M^{-1}&0_r\end{pmatrix}$
  on $(a_\ell,b_\ell)$, $\ell=0,\dots,g$.
\item $\mathcal N_+=\mathcal N_-\,e^{-in\omega_j\sigma_3}$ on $\Sigma_j$,
  $j=1,\dots,g$.
\item $\mathcal N(z)=I_{2r}+O(z^{-1})$ as $z\to\infty$.
\item $\mathcal N(z)=O\bigl(|z-e|^{-1/4}\bigr)$ as $z\to e$, $e\in\mathcal E$.
\end{enumerate}
\end{rhp}
 
Let $D$ be as in Lemma \ref{lem:varyingszego} and let $D^\sharp$, $\mathsf D$ and
$\mathsf D_\infty$ be defined by \eqref{eq:jacobisharp} and
\eqref{eq:jacobiconjugators}.  Item 5 is stated without the factor
$\mathsf D^{-1}$ that occurs in RH Problem \ref{rhp:jacobiouter}, because
$D^{\pm1}$ are bounded by Lemma \ref{lem:varyingszego}; this is the only formal
difference between the two global problems.
 
\begin{proposition}\label{prop:varyingreduction}
Let $U_1,\dots,U_g\in\Ur$ be as in \eqref{eq:gap-main} for $M$, let
$\omega_1,\dots,\omega_g$ be as in \eqref{eq:varyingomega}, and put
\begin{equation}\label{eq:varyingK}
  K_j=e^{-in\omega_j}U_j\in\Ur,\qquad j=1,\dots,g.
\end{equation}
If $N$ solves RH Problem \ref{def:rhp} with these data, then
\begin{equation}\label{eq:varyingglobal}
  \mathcal N(z)=\mathsf D_\infty\,N(z)\,\mathsf D(z)
\end{equation}
solves RH Problem \ref{rhp:varyingouter}, and conversely.  In particular RH
Problem \ref{rhp:varyingouter} has exactly one solution, by Propositions
\ref{prop:Nsolves} and \ref{prop:unique}.
\end{proposition}
 
\begin{proof}
The proof of Proposition \ref{prop:jacobireduction} applies word for word, with
$W$ replaced by $M$ throughout; the only property of the weight that is used is
$M=D_\pm D_\pm^*$ on the cuts together with $D_+=D_-U_j$ on the gaps, that is,
\eqref{eq:band-main} and \eqref{eq:gap-main}.
\end{proof}
 
Again only the local analysis remains.  Fix \(e\in\mathcal E\), and
let $Q$ be as in \cite[1.2]{DR25}.  Put
\[
 c_{e,\pm}=\phi_\pm(e),\qquad
 \phi_e(z)=\phi(z)-c_{e,\pm},\qquad \pm\im z>0.
\]
Following \cite[(5.24)]{DR25} and \ref{ass:model2}, write
\[
 P_e(z)=E_{e,n}(z)P_e^{(1)}(z)
 \begin{pmatrix}
 e^{\frac n2\phi(z)}Q(z)^{-1}&0\\
 0&e^{-\frac n2\phi(z)}Q^\sharp(z)
 \end{pmatrix},
 \qquad Q^\sharp(z)=Q(\bar z)^* .
\]
At an internal endpoint adjacent to \(\Sigma_j\), one has
\[
 \phi_+(e)-\phi_-(e)=-2i\omega_j.
\]
Consequently, conjugation by the last factor changes the jump of \(S\)
on \(\Sigma_j\cap\Delta_e\) into
\[
 \begin{pmatrix}I_r&I_r\\0&I_r\end{pmatrix}.
\]
Thus the diagonal jump \(e^{-in\omega_j\sigma_3}\) is already carried
by the exponential in the ansatz.  On the cut and the lens boundary the
same substitution gives the other two constant jumps in
\cite[(5.25)]{DR25}.

Since,
\[
 f_e(z)=\left(\frac34\phi_e(z)\right)^{2/3}
\]
is conformal near \(e\), the problem for \(P_e^{(1)}\) is
precisely the Airy problem \cite[(5.29)]{DR25}, and is solved in
\cite[\S5.5]{DR25} at a right endpoint and by the reflected
construction of \cite[\S5.6]{DR25} at a left endpoint.  The prefactor
\(E_{e,n}\) is determined by matching with \(\mathcal N\), as in
\cite[(5.37) and (5.50)]{DR25}; its apparent singularity at \(e\) is
removable by the same argument as in Lemma \ref{lem:Phi}.  Therefore
\[
 P_e(z)\mathcal N(z)^{-1}=I_{2r}+O(n^{-1}),
 \qquad z\in\partial\Delta_e.
\]
The final transformation is a small-norm RH problem, and reversing
\eqref{eq:varyingT} and \eqref{eq:varyingS} gives the 
asymptotic formula.

\begin{equation}\label{MVOP2}  e^{-ng(z)}P_n(z)
  =D(\infty)\,\cc(z)\,F_1\bigl(z;K_1,\dots,K_g\bigr)\,D(z)^{-1}+O(n^{-1}),
\end{equation}
uniformly on compact
subsets of $\C\setminus [a_0,b_g]$, which for $g=0$ is \cite[Theorem 1(i)]{DR25}.
\paragraph{Concluding remarks}:
We would like to highlight that in Assumption \ref{ass:model2} the matrix part of the weight in \eqref{eq:model2weight} is independent of $n$. One can consider $W_n(x)=W(x)^n$ as well. This comes with significant challenges, as the potential theory with $g$ function is now carried out on a Riemann surface determined by the spectral curve of the eigenvalues (see for e.g. \cite{K25}). The Deift-Zhou steepest analysis can also be posed in a Riemann-Surface. This is an on-going work with M. van Horrsen, A.B.J. Kuijlaars, and M. Piorkowski. We hope to come back to this soon.
\subsection*{Acknowledgements}
I thank Max van Horrsen for helpful discussions.
Sampad Lahiry acknowledges financial support from the International Research Training Group (IRTG) between KU Leuven and University of Melbourne and the  Melbourne Research Scholarship of University of Melbourne.

\end{document}